 \documentclass[draft]{article}

\usepackage{amsmath,amsfonts,amsthm,amssymb,amscd,color}
\renewcommand{\Re}{\mathop{\rm Re}\nolimits}

\theoremstyle{plain} \newtheorem{theorem}{Theorem}[section]
\newtheorem{lemma}[theorem]{Lemma}
\newtheorem{proposition}[theorem]{Proposition}
\newtheorem{corollary}[theorem]{Corollary} \theoremstyle{definition}
\newtheorem{definition}[theorem]{Definition} \theoremstyle{remark}
\newtheorem{remark}[theorem]{Remark}

\newcommand{\R}{{\mathbb R}}

\newcommand{\Z}{{\mathbb Z}}

\newcommand{\N}{{\mathbb N}}

\def\im{{\rm i}}

\newcommand{\C}{\mathbb{C}}

\newcommand{\E}{\mathcal{E}}

\def\({\left(}
\def\){\right)}
\def\<{\left\langle}
\def\>{\right\rangle}
\def\cprime{$'$}
\providecommand{\bysame}{\leavevmode\hbox to3em{\hrulefill}\thinspace}
\providecommand{\MR}{\relax\ifhmode\unskip\space\fi MR }
\providecommand{\href}[2]{#2}

\numberwithin{equation}{section}

\begin{document}

\title{Existence and asymptotic stability of quasi-periodic solutions of discrete NLS with potential}

\author {Masaya Maeda}

\maketitle

\begin{abstract}
We prove the existence of a 2-parameter family of small quasi-periodic solutions of discrete nonlinear Schr\"odinger equation (DNLS).
We further show that all small solutions of DNLS decouples to one of these quasi-periodic solutions and dispersive wave.
As a byproduct, we show that all small nonlinear bound states including excited states are unstable.
\end{abstract}

\section{Introduction}
In this paper, we consider small solutions of the discrete nonlinear Schr\"odinger equation (DNLS):
\begin{align}\label{1}
\im \partial_t u=H u + |u|^{6}u, \quad u:\R\times \Z\to \C,
\end{align}
where, $H=-\Delta + V$, $\Delta$ is the discrete Laplacian:
\begin{align*}
(\Delta u)(n):=u(n+1)-2u(n) + u(n-1),
\end{align*}
and $\sum_{n\in\Z}(1+|n|^2)^{1/2}|V(n)|<\infty$.
We assume that
$\sigma_d(H)=\{e_1<e_2\}$ with  
\begin{align}\label{2}
e_1+n(e_2-e_1)\notin [0,4],\quad \forall n\in\Z,
\end{align}
where $\sigma_d(H)$ is the set of discrete spectrum of $H$.
Further, set $\phi_1$, $\phi_2$ to be the normalized real valued eigenfunctions associated to $e_1,e_2$ respectively.

The aim of this paper is to study the long time behavior of small solutions of DNLS \eqref{1}.
Before explaining our results, we briefly recall the known results for the ``continuous" nonlinear Schr\"odinger equations (NLS):
\begin{align*}
\im u_t=H_c u+|u|^2u,\quad u:\R\times \R^3\to \C.
\end{align*}
Here, we set $H_c=-\Delta+V$, $V$ is a Schwartz function and assume $\sigma_d(H)=\{e_1<e_2\}$ with $e_2<0$.
In this case, it is known that all small solutions decouple into a nonlinear bound state and dispersive wave \cite{CuMaAPDE, SW04RMP, TY02ATMP, TY02IMRN, TY02CPDE}.
Here, a nonlinear bound state is a time periodic solution with the form $e^{-\im \omega t}\phi_\omega(x)$ and a dispersive wave is a solution which tends to $0$ in $L^\infty$ (or $l^\infty$ in the discrete case) as $t\to \infty$.
In particular, since dispersive wave vanishes locally, we see that all solutions locally converges to some nonlinear bound state.
Because the linear Schr\"odinger equation has quasi-periodic  solutions such as $a_1 e^{-\im e_1 t}\phi_1 + a_2 e^{-\im e_2 t}\phi_2$, it is striking that NLS has no small quasi-periodic solutions.
The mechanism which prevent the existence of quasi-periodic solution is due to the interaction between the continuous spectrum and the discrete spectrum.
In particular, when the frequencies $e_1+n(e_2-e_1)$ hit the continuous spectrum, there is a damping from the discrete spectrum to the continuous spectrum $[0,\infty)$. 

We now come back to the discrete case.
For the discrete case there is a possibility that the frequencies $e_1+n(e_2-e_1)$ never hit the continuous spectrum since the spectrum of the discrete Laplacian is $[0,4]$.
This is assumption \eqref{2}.
In this case, there is no nonlinear interaction between the continous spectrum and the discrete spectrum.
Thus, we can expect there may exists a quasi-periodic solution.
Indeed, in this paper we show the existence of quasi-periodic solutions in the form $\Psi(z_1,z_2)\sim z_1\phi_1+z_2\phi_2$, parametrized by small complex parameters $z_1,z_2$ (see, Theorem \ref{thm:1}).
Using this family of quasi-periodic solutions, we also show that all solutions decouples into this quasi-periodic solution and dispersive wave (Theorem \ref{thm:2}).

We now prepare some notations to state our results precisely.

\begin{itemize}
\item
For $p\geq 1$, $\sigma\in\R$, we set
$l^{p,\sigma}(\Z):=\left\{ u=\{u(n)\}_{n\in\Z}\ |\ \|u\|_{l^{p,\sigma}}^p:=\sum_{n\in\Z}\<n\>^{p\sigma}|u(n)|^p<\infty\right\},$ where $\<n\>:=(1+n^2)^{1/2}$.
Further, $l^p(\Z):=l^{p,0}(\Z)$.
\item
We define the inner-product of $l^2(\Z)$ by
$
\<u,v\>:=\Re\sum_{n\in \Z}u(n)\overline{v(n)}.
$
\item
For $a\in \R$, we set
$l^a_e(\Z):=\{u=\{u(n)\}_{n\in\Z}\ |\ \|u\|_{l^a_e}^2:=\sum_{n\in\Z} e^{2a|n|}|u(n)|^2<\infty\}.
$
\item
We often write $a\lesssim b$ by meaning that there exists a constant $C$ s.t.\  $a\leq Cb$.
If we have $a\lesssim b$ and $b\lesssim a$, we write $a\sim b$.
\item
For a Banach space $X$ equipped with the norm $\|\cdot\|_X$,
we set
$
B_X(\delta):=\{u\in X\ |\ \|u\|_{X}<\delta\}.
$
\item
For Banach spaces $X,Y$, we set $\mathcal L(X;Y)$ to be the Banach space of all bounded operators from $X$ to $Y$, and $\mathcal L(X):=\mathcal L(X;X)$.
Further, we set $\mathcal L^n(X;Y)$ inductively by $\mathcal L^n(X;Y)=\mathcal L(X;\mathcal L^{n-1}(X;Y))$ and $\mathcal L^0(X;Y)=Y$.
\item
We set $C^\omega(B_X(\delta);Y)$ to be all real analytic functions from $B_X(\delta)$ to $Y$.
By real analytic functions, we mean that $f:B_X(\delta)\to Y$ can be written as $f(x)=\sum_{n\geq 0}a_n x^n$ with $\sum_{n\geq 0}\|a_n\|_{\mathcal L^n(X;Y)}r^n<\infty$ for all $r<\delta$, where $a_n\in \mathcal L^n(X;Y)$ and $a_nx^n:=a_n(x, x, \cdots, x)$.

\end{itemize}

It is well known that there exist families of nonlinear bound states of \eqref{1}.
For the convenience of the readers, we will give the proof in the appendix of this paper.

\begin{proposition}\label{prop:1}

Fix $j\in \{1,2\}$.
There exist $a_0>0$ and $\delta_0>0$ s.t.\ for all $z\in B_{\C}(\delta_0)$, there exists $\tilde e_j\in C^\omega\(B_{\R}(\delta_0^2); \R\)$ and $q_j\in C^\omega\(B_{\R}(\delta_0^2); l^{a_0}_e(\Z;\R)\)$ s.t.\ $\<\phi_j, q_j\>=0$ and 
\begin{align}\label{3}
\phi_j(z):=z\tilde\phi_j(|z|^2)=z\(\phi_j+q_j(|z|^2)\),
\end{align} satisfies
\begin{align}\label{4}
\(H-E_j(|z|^2)\)\phi_j(z)+|\phi_j(z)|^6\phi_j(z)=0,
\end{align}
where $E_j(|z|^2)=e_j+\tilde e_j(|z|^2)$.
Further, we have $|\tilde e_j(|z|^2)|+\|q_j(|z|^2)\|_{l_e^{a_0}}\lesssim |z|^6$.
\end{proposition}

\begin{remark}
Notice that if $\phi$ satisfies \eqref{4}, then $e^{-\im E_j t}\phi$ is the solution of \eqref{1}.
\end{remark}

The first result of this paper is the existence of quasi-periodic solutions of \eqref{1}.
\begin{theorem}\label{thm:1}
There exist $a_1\in(0,a_0)$ and $\delta_1\in (0,\delta_0)$ s.t.\ there exist
$
\psi\in C^\omega\(B_{\C^2}(\delta_1); l^{a_1}_e(\Z;\C)\)$
and
$
\varepsilon_j\in C^\omega\(B_{\R^2}(\delta_1^2); \R\)$ for $j=1,2$,
s.t.
\begin{align*}
\Psi(z_1,z_2):=\phi_1(z_1)+\phi_2(z_2)+\psi(z_1,z_2),
\end{align*}
is a solution of \eqref{1} if $z_j$ ($j=1,2$) satisfies
\begin{align}\label{4.0}
\im \dot z_j = \(E_j(|z_j|^2)+\varepsilon_j(|z_1|^2,|z_2|^2)\)z_j.
\end{align}
Further, for arbitrary $\theta\in\R$, we have
\begin{align}\label{4.001}
e^{\im \theta}\psi(z_1,z_2)=\psi(e^{\im \theta}z_1,e^{\im \theta}z_2),
\end{align}
and
\begin{align}
&\|\psi(z_1,z_2)\|_{l_e^{a_1}}\lesssim |z_1||z_2| (|z_1|^5+|z_2|^5)\label{4.01},\\&
|\varepsilon_j(|z_1|^2,|z_2|^2)|\lesssim |z_{3-j}|^2\(|z_1|^4+|z_2|^4\).\label{4.02}
\end{align}
\end{theorem}

The second result of this paper is about the asymptotic behavior of small solution of \eqref{1}.
%In particular, we show that all solutions of \eqref{1} with $\|u(0)\|_{l^2}$ small, decomposes to the quasi-periodic solution $\Psi$ obtained in Theorem \ref{thm:1} and free linear solution of $\im u_t = -\Delta u$ as $t\to \infty$.

\begin{theorem}\label{thm:2}
Assume $H$ is generic (for the definition see Lemma 5.3 of \cite{CT09SIAM}).
Then, there exists $\delta_2\in (0,\delta_1)$ s.t.\ if $\|u_0\|_{l^2}<\delta_2$, then the solution of \eqref{1} with $u(0)=u_0$ exists globally in time and
there exist $z_j(t):[0,\infty)\to \C$, $\rho_{j,+}\in \R_{\geq 0}$ for $j=1,2$ and $v_+\in l^2$ s.t.
\begin{align*}
\lim_{t\to\infty}\|u(t)-\Psi(z_1(t),z_2(t))-e^{\im t \Delta}v_+\|_{l^2}= 0,\quad \lim_{t\to \infty}|z_j(t)|=\rho_{j,+},\ (j=1,2).
\end{align*}
Further, we have $\|v_+\|_{l^2}+\rho_{1,+}+\rho_{2,+}\lesssim \|u_0\|_{l^2}$.
\end{theorem}

\begin{remark}
Theorem \ref{thm:1} actually holds even if we replace the nonlinearity $|u|^6 u$ to $|u|^{2p}u$ for $p\in \N$.
However, for Theorem \ref{thm:2}, we need $p\geq 3$.
For simplicity, we decided only to consider the case $p=3$.
The assumption for $H$ is used for the linear estimates of $e^{\im t H}$.
See section \ref{sec:linest}.
\end{remark}

As a corollary of Theorem \ref{thm:2}, we have orbital stability of nonlinear bound states $\phi_j(z)$.
Here, for fixed $j$ and $z$, we say $\phi_j(z)$ is orbitally stable if for all $\varepsilon>0$, there exists $\delta>0$ s.t.\ if $\|u(0)-\phi_j(z)\|_{l^2}$, then $\sup_{t>0}\inf_{\theta\in\R}\|u(t)-e^{\im \theta}\phi_j(z)\|_{l^2}<\varepsilon$.

\begin{corollary}\label{cor:1}
Under the assumptions of Theorem \ref{thm:2}, for $j=1,2$ and $|z|\lesssim \delta_2$, $\phi_j(z)$ is orbitally stable.
\end{corollary}

\begin{remark}
For the case $e_1<0$, one can show $\phi_1(z)$ is trapped by the energy (i.e.\ $\phi_1(z)$ is a minimizer of the energy (given in \eqref{9.01}) under the constraint $\|u\|_{l^2}=\|\phi_1(z)\|_{l^2}$).
By a classical argument by Cazenave-Lions \cite{CL82}, one can conclude that $\phi_1(z)$ is orbitally stable.
Similarly, if $e_2>4$, $\phi_2(z)$ is a maximizer of the energy under the constraint $\|u\|_{l^2}=\|\phi(z)\|_{l^2}$, and we can show it is orbitally stable.
However, the case $e_1<e_2<0$ is interesting.
In this case $\phi_2(z)$ is not trapped by the energy, which means that $\phi_2(z)$ is not a minimizer (nor a maximizer) of energy $E$ under the constraint $\|u\|_{l^2}=\|\phi(z)\|_{l^2}$.
Therefore, in this case one cannot show the orbital stability by variational methods.
\end{remark}

We now recall the known results related to our results on continuous and discrete NLS.
There is a long list of papers on asymptotic stability of both large and small nonlinear bound states of NLS \cite{Bambusi13CMP,BP92StP,BP92Ast,BP92AA,BP95,BS03AIHPN,Cuccagna01CPAM,Cuccagna03RMP,Cuccagna06JDE,Cuccagna08JDE,Cuccagna11CMP,Cuccagna14TAMS,
CuMaAPDE,CM08CMP,CP14AA,CT09AIHPN,GS05RMP,GNT04,KM09JFA,KZ07CMP,KZ09JDE,Koo11JDE,
Miz07JMKU,Miz08JMKU,
Perelman04CPDE,PW97JDE,SW90CMP, SW92JDE,SW04RMP,TY02ATMP,TY02IMRN,TY02CPDE}.

The asymptotic stability for small nonlinear bound states of NLS was first proved by Soffer-Weinstein \cite{SW90CMP}.
They assume that the Schr\"odinger operator $H_c=-\Delta+V$ has exactly one eigenvalue and the initial data is small in some weighted space.
Later, Gustafson-Nakanishi-Tsai \cite{GNT04} proved the asymptotic stability in the energy space $H^1$ for the $3$ dimensional case.
%They proved that if the spatial dimension is $3$ and $H=-\Delta+V$ has only one eigenvalue, then all small solution in the energy space decomposes to standing wave and free solution of linear Schr\"odinger equation as $t\to \infty$.
One of the main tool of \cite{GNT04} was the endpoint Strichartz estimate \cite{GoSc04,KT98AJM} which collapse in the $1$ and $2$ dimensional cases.
For $1$ and $2$ dimensional cases, Mizumachi \cite{Miz07JMKU,Miz08JMKU} prove the asymptotic stability result in the energy space by replacing the endpoint Strichartz estimates to Kato type smoothing estimates.
The results \cite{GNT04, Miz07JMKU,Miz08JMKU} tells us that, under the assumption that $H_c$ has exactly one eigenvalue, the dynamics of small solutions of NLS is similar to the linear Schr\"odinger equation .
This is because the solution of linear Schr\"odinger equation also decomposes to a periodic solution associated to the eigenvalue and dispersive wave associated to the absolutely continuous spectrum.

The situation changes drastically when $H_c$ has more than two eigenvalues.
Indeed, in this case there exist quasi-periodic solutions of the linear Schr\"odinger equation associated to the two eigenvalues of $H_c$.
However, \cite{Sigal93CMP} proved that there exists no small quasi-periodic solution of NLS.
Further, \cite{SW04RMP}, \cite{TY02IMRN}, \cite{TY02CPDE}, \cite{TY02ATMP} proved that if $H_c$ has two eignvalues with $e_1<2e_2$, all small solutions in some weighted space decomposes to a nonlinear bound state and dispersive wave.
Recently, \cite{CuMaAPDE} extended these result to the case $H_c$ has more than two eigenvalues and removed the assumption $e_1<2e_2$.
See also related results for nonlinear Klein-Gordon equation (NLKG) \cite{BC11AJM,CuMaP,SW99IM} and nonlinear Dirac equation \cite{CTpreDirac}.
The mechanism which prevents the existence of quasi-periodic solutions is the nonlinear interaction between the eigenvalue and the absolutely continuous spectrum.
The non-degeneracy condition for such interaction is called Fermi Golden Rule (FGR) which all the above papers assume.

We now turn to the known results of DNLS.
For the case that the discrete Schr\"odinger operator $H$ has only one eigenvalue, \cite{CT09SIAM,KPS09SIAM} proved the asymptotic stability result in the energy space $l^2$.
See also \cite{MP12DCDS} for asymptotic stability results in weighted space for lower power nonlinearity.
This result corresponds to the continuous case.
However, for the case $H$ has two eigenvalues with $e_1<0<4<e_2$, \cite{Cuccagna10DCDS} proved that the ground state (which is $\phi_1(z)$ in Proposition \ref{prop:1}) is orbitally stable but not asymptotically stable.
For the continuous case, ground state is asymptotically stable, so this result shows that in this case the small solution of continuous and discrete NLS has different asymptotic dynamics.
As mentioned in \cite{Cuccagna10DCDS}, the situation that the nonlinear bound state is orbitally stable but not asymptotically stable suggests that there exist quasi-periodic solutions.
Indeed, Theorem \ref{thm:1} shows that there exists a $2$-parameter family of quasi-periodic solutions which bifurcates from the two eigenvalues of $H$.
Note that the fact that the standing wave is not asymptotically stable is a direct consequence of the existence of quasi-periodic solution near standing waves.

Up to here, we have only discussed the nonlinear bound states and quasi-periodic solutions which bifurcate from the eigenvalues of the Schr\"odinger operator.
We note that it is known that there exist different kinds of periodic and quasi-periodic solutions for DNLS (mainly considered in the translation invariant case, i.e.\ $V\equiv 0$).
First, if the nonlinearity is attractive, there exists a nonlinear bound state which can be approximated by the nonlinear bound state of the continuous NLS \cite{BP10N} (See also \cite{BPP10AA}).
Second, in the ``anticontinous limit" (which is the situation we are putting $\varepsilon\ll1$ in front of $\Delta$), there exists quasi-periodic solutions (See for example \cite{Bambusi13CMP, JA97N, MA94N}).
After rescaling, the quasi-periodic solutions of this kind will have large amplitude.

We prove the existence of the quasi-periodic solutions starting from assuming that the quasi-periodic solution can be written as $\sum_{n\in \Z}e^{-\im (\mathcal E_1 + n(\mathcal E_2-\mathcal E_1))t}v_n$, where $\mathcal E_j\sim e_j$ and solve \eqref{1} for each frequency.
Notice that the frequencies $\{ \mathcal E_1 + n (\mathcal E_2-\mathcal E_1)\}_{n\in \Z}$ are generated from the two standing waves and the nonlinearity.
Further, there is no intersection between these frequencies and the continuous spectrum of $H$ because of \eqref{2}.
This assumption is crucial for the existence of quasi-periodic solution.
Indeed, for the continuous NLS case, condition \eqref{2} always fails because the continuous spectrum is $[0,\infty)$.
Then, by the nonlinear interaction, we have a damping from the point spectrum to the continuous spectrum which prevents the existence of the quasi-periodic solutions.
By the same reason, we conjecture that for the case $H$ has more than $3$ eigenvalues there will be no quasi-periodic solution like
\begin{align*}
\Psi(z_1,z_2,z_3)\sim z_1\phi_1+z_2\phi_2+z_3\phi_3.
\end{align*}
This is because the nonlinear interaction between the point spectrum and absolutely continuous spectrum arises again and there will be a damping.

For the asymptotic stability result Theorem \ref{thm:2}, we start from a standard modulation argument and adapt the nonlinear coordinate given in \cite{GNT04}.
However, since our quasi-periodic solution is not a standing wave, it seems to be difficult to get a simple equations for the modulation parameters in this coordinate.
To overcome this difficulty,
we use the Darboux theorem which was introduced in \cite{Cuccagna11CMP} and used in \cite{Bambusi13CMP, CuMaAPDE,CM14JDE, CTpreDirac}.
%In these works, Darboux theorem was used to make it possible to do the Birkoff normal form argument.
In fact, after changing the coordinates by the Darboux theorem, we get a well decoupled equations (see \eqref{40}, \eqref{41}) which are easy to analyze.
We note that although we have made the change of coordinate with a real analytic regularity, we actually need only $C^3$.
The real analyticity comes from the real analyticity of the nonlinearity.
Therefore, for the asymptotic stability, we do not need real analyticity.
However, for the existence of the quasi-periodic solution, we can only handle a polynomial nonlinearity because we have expanded the solution as $\sum_{n\in \Z}e^{-\im (\mathcal E_0 + n(\mathcal E_1-\mathcal E_0))t}v_n$.
Further, real analyticity reduces the amount of some computations for the estimate of the derivatives of the coordinate change (see Lemma \ref{lem:15}).
These are the reasons why we have adapted the real analytic framework for the change of coordinate.

We now mention about the difference between the proof of \cite{Bambusi13CMP} whcih shows the asymptotic stability of periodic solutions obtained by the anti-continuous limit (which is a large solutions) and the proof of Theorem \ref{thm:2}.
The difference is that \cite{Bambusi13CMP} uses the normal form argument infinite times (the Birkhoff normal form).
For this method, it is necessary to have the analyticity of the nonlinear term for the convergence of the normal form steps.
On the other hand, we only use the normal form argument (the Darboux theorem) once.
As mentioned before, our argument only requires $C^3$ regularity for the coordinate change so it is not necessary to have a analytic nonlinearity for the proof of asymptotic stability.
However, we need the nonlinearity to be polynomial for the proof of the existence of the quasi-periodic solution.
%The corresponding method of \cite{Bambusi13CMP} in our situation seems to be starting from the linear coordinate (the spectral coordinate of $H$) and using the normal form argument infinite times to decouple the system.
%This method gives the existence and asymptotic stability in once but it is unclear that if it works in our case.
%We use the linear estimates in \cite{CT09SIAM, KPS09SIAM,PS08JMP}.

The paper is organized as follows:
In section \ref{sec:existence}, we prove Theorem \ref{thm:1}.
In section \ref{sec:asymp}, following \cite{GNT04}, we set up the nonlinear coordinate.
In section \ref{sec:darboux}, we prove the Darboux theorem and rewrite DNLS in the new coordinate, the new system is given in \eqref{40}-\eqref{41}.
In section \ref{sec:linest}, we introduce the linear estimates which were originally given in \cite{CT09SIAM} and in section \ref{sec:proofthm2}, we prove Theorem \ref{thm:2}.
In the appendix we give the proof of Proposition \ref{prop:1}, Lemma \ref{lem:2.0}.

%For $j=1,2$, $A=R,I$,
%we set $D_{j,A}f:=\partial_{z_{j,A}}f$.
%Further, for a multi-index $\mathbf m=(m_1,\cdots,m_N)$, with $m_k\in \{(j,A)\ |j=1,2,A=R,I\}$, we set $|\mathbf m|=N$ and
%$D^{\mathbf m}=D_{m_1}D_{m_2}\cdots D_{m_N}.$

\section{Proof of Theorem \ref{thm:1}}\label{sec:existence}

In this section, we construct solutions of \eqref{1} under the following ansatz:
\begin{align}
\Psi(z_1,z_2)=\phi_1(z_1)+\phi_2(z_2)+\sum_{m\geq 0}\( z_1^{m+1}\overline{z_2}^mv_{1m}+ \overline{z_1}^mz_2^{m+1}  v_{2m}\),\label{4.1}
\end{align}
where, $v_{1m},v_{2m}$ are real valued and $\<v_{j0},\phi_j\>=0$ for $j=1,2$.
\begin{remark}
Notice that if $v_{jm}=v_{jm}(|z_1|^2,|z_2|^2)$, then we have $\Psi(e^{\im \theta}z_1,e^{\im \theta}z_2)=e^{\im \theta}\Psi(z_1,z_2)$.
\end{remark}
Since we want to reduce the problem of construction of quasi-periodic solution to the construction of solution of system elliptic equations, we assume that for $\varepsilon_j\in \R$ given below (see, \eqref{6}), $z_j$ ($j=1,2$) satisfies
\begin{align}
\im \dot z_j = \mathcal E_j z_j,
\end{align}
where $\mathcal E_j=E_j(|z_j|^2)+\varepsilon_j=e_j+\tilde e_j(|z_j|^2)+\varepsilon_j$.
Then, we have
\begin{align*}
\im \partial_t \Psi(z_1,z_2)=&\sum_{j=1,2}\mathcal E_jz_j\(\tilde \phi_j(|z_j|^2)+v_{j0}\)\\&
+\sum_{m\geq 1}z_1^{m+1}\overline{z_2}^m\((m+1)\mathcal E_1-m \mathcal E_2\)v_{1m}+\overline{z_1}^mz_2^{m+1}\((m+1)\mathcal E_2-m \mathcal E_1\)v_{2m},\\
H\Psi(z_1,z_2)=&\sum_{j=1,2}z_j\(H\tilde\phi_j(|z_j|^2)+H v_{j0}\)+\sum_{m\geq 1}\(z_1^{m+1}\overline{z_2}^m H v_{1m} +  \overline{z_1}^mz_2^{m+1} H v_{2m}\),
\end{align*}
where $\tilde \phi_j(|z_j|^2)$ is given in \eqref{3}.
Further, for $\mathbf v=\{v_{j,m}\}_{j=1,2,m\geq 0}$, we have
\begin{align*}
|\Psi(z_1,z_2)|^6\Psi(z_1,z_2)=|\phi_1(z_1)|^6\phi_1(z_1)+|\phi_2(z_2)|^6\phi_2(z_2)+\mathcal N(|z_1|^2,|z_2|^2, \mathbf v),
\end{align*}
where
\begin{align}
\mathcal N(|z_1|^2,|z_2|^2, \mathbf v)=\sum_{m\geq 0}z_1^{m+1}\overline{z_2}^m N_{1m}(|z_1|^2,|z_2|^2,\mathbf v)
+\sum_{m\geq 0}\overline{z_1}^mz_2^{m+1}  N_{2m}(|z_1|^2,|z_2|^2,\mathbf v),\label{4.2}
\end{align}
for some $\{N_{jm}\}_{j=1,2,m\geq 0}$ (see Lemma \ref{lem:2.0} below).
Therefore, to construct a solution of \eqref{1} in the form \eqref{4.1}, it suffices to solve the system of  equations of the coefficients of $z_1^{m+1}\overline{z_2}^m$ and $\overline{z_1}^mz_2^{m+1}$.
In particular, we solve the system of elliptic equations
\begin{align}\label{5}
% (H -e_1)v_0&=\varepsilon_1\tilde\phi_1(|z_1|^2)+(\tilde e_1(|z_1|^2)+\varepsilon_1)v_0-N_0  ,\\\nonumber
% (H -e_2)w_0&=\varepsilon_2\tilde\phi_2(|z_2|^2)+(\tilde e_2(|z_2|^2+\varepsilon_2)w_0 -M_0 ,\\\nonumber
\(H-\omega_{jm}\)v_{jm}=&\delta_{0m}\varepsilon_j\tilde \phi_j(|z_j|^2)+\delta(j) m(\tilde e_1(|z_1|^2)-\tilde e_2(|z_2|^2)+\varepsilon_1-\varepsilon_2)v_{jm}\nonumber\\&+(\tilde e_j(|z_j|^2)+\varepsilon_j)v_{jm} -N_{jm},
%\(H-\omega_{2,m}\) v_{2,m}&=\((m+1)(\tilde e_2(|z_2|^2)+\varepsilon_2)-m(\tilde e_1(|z_1|^2)+\varepsilon_1)\)v_{2,m} -N_{2,m}  ,\ (m\geq 1),
\end{align}
where $\delta_{0m}=1$ if $m=0$ and $0$ otherwise, $\delta(1)=1$, $\delta(2)=-1$ and $\omega_{1m}=(m+1)e_1-m e_2$ and $\omega_{2m}=(m+1)e_2-me_1$.
Let $Q_jv=\<v,\phi_j\>\phi_j$.
By applying $Q$ to \eqref{5} with $m=0$ , we have
\begin{align}\label{6}
\varepsilon_j=\varepsilon_j(|z_1|^2,|z_2|^2,\mathbf v)=\<N_{j0}(|z_1|^2,|z_2|^2,\mathbf v),\phi_j\>.
\end{align}
Therefore, it suffices to solve
\begin{align}
(H-\omega_{jm})v_{jm}=\delta_{0m}\varepsilon_j q_j+\delta(j) m(\tilde e_1-\tilde e_2+\varepsilon_1-\varepsilon_2)v_{jm}+(\tilde e_j+\varepsilon_j)v_{jm} -(1-\delta_{0m}Q_j)N_{jm},\label{6.0}
\end{align}
where $\varepsilon_j$ is now given by \eqref{6}.

Since we want to solve the system \eqref{6.0} by fixed point argument, we define a function space $X_{ar}$ for $a,r>0$ by
\begin{align*}
&X_{ar}:=\{\mathbf v=\{v_{jm}\}_{j=1,2,m\geq 0}\ |\ v_{jm}\in l_e^a,\ \|\mathbf v\|_{ar}:=\sum_{j=1,2,m\geq 0}r^{2m+1}\|v_{jm}\|_{l_e^a}<\infty\}.
\end{align*}
For $\mathbf v=\{v_{jm}\}_{j=1,2,m\geq 0}$, we set
\begin{align*}
\mathcal P \mathbf v:=\{(1-\delta_{0m}Q_j)v_{jm}\}_{j=1,2,m\geq 0},
\end{align*}
and define $X_{ar}^c:=\mathcal P X_{ar}$.
We next define the operator $\mathcal A, \mathcal B$
 on $X_{ar}^c$ by
\begin{align*}
\mathcal A \mathbf v&=\{(H-\omega_{jm})^{-1}v_{jm}\}_{j=1,2,m\geq 0},\\
\mathcal B \mathbf v&=\{\delta(j)m(H-\omega_{jm})^{-1}v_{jm}\}_{j=1,2,m\geq 0},
\end{align*}
where
$\mathbf v=\{v_{jm}\}_{j=1,2,m\geq 0}$.
Notice that $(H-e_j)$ are invertible on $(1-Q_j)l_e^a$ (see Lemma \ref{a1}).

\begin{lemma}\label{lem:3}
For sufficiently small $a>0$, we have
$\mathcal A,\mathcal B \in \mathcal L(X_{ar}^c,X_{ar}^c)$.
\end{lemma}

\begin{proof}
As written above, $\mathcal A$ and $\mathcal B$ are well defined on $X_{ar}^c$.
By Lemma \ref{lem:a1}, we have $$\|(1+m)(H-\omega_{jm})^{-1}v_{jm}\|_{l_e^a}\lesssim \|v_{jm}\|_{l_e^a}.$$
Therefore, we see that $\mathcal A$ and $\mathcal B$ are bounded on $X_{ar}^c$.
\end{proof}

Using the above notations, we can rewrite \eqref{6.0} as
\begin{align}\label{6.01}
\mathbf v = \mathbf \Phi (|z_1|^2,|z_2|^2,\mathbf v),
\end{align}
where
\begin{align}\label{6.02}
&\mathbf \Phi(|z_1|^2,|z_2|^2,\mathbf v)= \mathcal A \mathbf q(|z_1|^2,|z_2|^2,\mathbf v)+\sum_{l=1,2}\(\tilde e_l(|z_l|^2)+\varepsilon_l(|z_1|^2,|z_2|^2,\mathbf v)\)\mathbf 1_l \mathcal A \mathbf v \\&\quad+ \(\tilde e_1(|z_1|^2)-\tilde e_2(|z_2|^2)+\varepsilon_1(|z_1|^2,|z_2|^2,\mathbf v)-\varepsilon_2(|z_1|^2,|z_2|^2,\mathbf v)\)\mathcal B \mathbf v-\mathcal A \mathcal P \mathcal N(|z_1|^2,|z_2|^2,\mathbf v),\nonumber
\end{align}
where $\mathbf q(|z_1|^2,|z_2|^2,\mathbf v) =\{\delta_{0m}\varepsilon_j(|z_1|^2,|z_2|^2,\mathbf v) q_j(|z_j|^2)\}_{j=1,2,m\geq 0}$ and $\mathbf 1_l \mathbf v=\{\delta_{lj}v_{jm}\}_{j=1,2,m\geq 0}$.

%
%
%and
%\begin{align*}
%\mathcal T(z_1,z_2) \mathbf v:=\sum_{m\geq 0} z_1^{m+1}\overline{z_2}^m v_m+\sum_{m\geq 0} \overline{z_1}^mz_2^{m+1} w_m.
%\end{align*}

%We first show that if $\mathbf v\in X_{r,a}$, then $\mathcal T(z_1,z_2)\mathbf v\in l_e^{a}$ if $|z_1|, |z_2|\leq r$.
%
%\begin{lemma}\label{lem:1}
%Let $|z_1|,|z_2|\leq r$ and $\mathbf v\in X_{r, a}$, then 
%\begin{align}\label{6}
%\|\mathcal T(z_1,z_2)\mathbf v\|_{l_e^{a}}\leq \|\mathbf v\|_{r, a}.
%\end{align}
%\end{lemma}
%
%\begin{proof}
%Set $\mathbf v=\left\{\{v_m\}_{m\geq 0}, \{w_m\}_{m\geq 0}\right\}$.
%Then, we have
%\begin{align*}
%&\|\mathcal T(z_1,z_2)\mathbf v \|_{l_e^{a}}\leq \sum_{m\geq 0}|z_1|^{m+1}|z_2|^m\|v_m\|_{l_e^{a}}+ |z_1|^m|z_2|^{m+1}\|w_m\|_{l_e^{a}}\leq \|\mathbf v\|_{\alpha,r}.
%\end{align*}
%Therefore, we have the conclusion.
%\end{proof}
%
%\begin{remark}
%By the above lemma, we see that if $\mathbf v(\cdot,\cdot)\in C^\omega(B_{\R^2}(0,\delta);X_{r,a})$ for $\delta<r^2$, then $\Psi(z_1,z_2):=\mathcal T(z_1,z_2)\mathbf v(|z_1|^2,|z_2|^2)\in l_e^a$ is real analytic with respect to $z_{1,R}, z_{1,I}, z_{2,R}, z_{2,I}$.
%\end{remark}

To express $\mathcal N$, we introduce the following multilinear operator on $X_{ar}$.
\begin{definition}
For $\mathbf v^k=\{v_{jm}^k\}_{j=1,2,m\geq 0}$, $k=1,2,3$, we define $\mathcal M(|z_1|^2,|z_2|^2,\mathbf v^1,\mathbf v^2,\mathbf v^3)=\{M_{jm}(|z_1|^2,|z_2|^2,\mathbf v^1,\mathbf v^2,\mathbf v^3)\}_{j=1,2,m\geq 0}$ by the relation
\begin{align}
&\sum_{m\geq 0}\( z_1^{m+1}\overline{z_2}^mM_{1m}+ \overline{z_1}^mz_2^{m+1}  M_{2m}\)=
\sum_{m_1\geq 0}\( z_1^{m_1+1}\overline{z_2}^{m_1}v_{1m_1}^1+ \overline{z_1}^{m_1}z_2^{m_1+1}  v_{2m_1}^1\)\nonumber\\&
\quad\quad \times\overline{\sum_{m_2\geq 0}\( z_1^{m_2+1}\overline{z_2}^{m_2}v_{1m_2}^2+ \overline{z_1}^{m_2}z_2^{m_2+1}  v_{2m_2}^2\)}\sum_{m_3\geq 0}\( z_1^{m_3+1}\overline{z_2}^{m_3}v_{1m_3}^3+ \overline{z_1}^{m_3}z_2^{m_3+1}  v_{2m_3}^3\).\label{6.03}
\end{align}
We inductively define $\mathcal M_{2k+1}(|z_1|^2,|z_2|^2,\mathbf v_1,\cdots,\mathbf v_{2k+1})$ for $k\geq 1$ by $\mathcal M_3=\mathcal M$ and
\begin{align*}
\mathcal M_{2k+1}(|z_1|^2,|z_2|^2,\mathbf v^1,\mathbf v^2,\mathbf v^{3})&=\mathcal M(|z_1|^2,|z_2|^2,\mathbf v^1,\mathbf v^2,\mathcal M_{2k-1}(|z_1|^2,|z_2|^2,\mathbf v^{3},\cdots,\mathbf v^{2k+1})),
\end{align*}
and
$\mathcal M_{2k+1}(|z_1|^2,|z_2|^2,\mathbf v):=\mathcal M_{2k+1}(|z_1|^2,|z_2|^2,\mathbf v,\cdots,\mathbf v)$.
\end{definition}

\begin{lemma}\label{lem:2.0}
Let $\delta< r$.
Then, we have $\mathcal M_{2k+1}\in C^\omega(B_{\R^2}(\delta^2);\mathcal L^{2k+1}(X_{ar};X_{ar}))$ and
\begin{align*}
\sup_{(z_1,z_2)\in B_{\C^2}(\delta)}\|\mathcal M_{2k+1}(|z_1|^2,|z_2|^2,\cdot)\|_{\mathcal L^{2k+1}(X_{ar};X_{ar})}\lesssim 1.
\end{align*}
\end{lemma}

We prove Lemma \ref{lem:2.0} in the appendix of this paper.

We set $\Phi_l(|z|^2)= \{\delta_{jl}\delta_{0m}\tilde \phi_j(|z|^2)\}_{j=1,2,m\geq 0}$.
Using, $\mathcal M_7$, we have
\begin{align}\label{7}
\mathcal N(|z_1|^2,|z_2|^2,\mathbf v)=&\mathcal M_7(|z_1|^2,|z_2|^2,\Phi_1(|z_1|^2)+\Phi_2(|z_2|^2)+\mathbf v)\\&-\mathcal M_7(|z_1|^2,|z_2|^2,\Phi_1(|z_1|^2))-\mathcal M_7(|z_1|^2,|z_2|^2,\Phi_2(|z_2|^2))\nonumber
\end{align}

We set $\mathcal C_j$ by $\mathcal C_j \mathbf v =\<\phi_j,v_{j0}\>$.
Then, since $|\mathcal C_j \mathbf v|\leq \|v_{j0}\|_{l_e^a}\leq r^{-1}\|\mathbf v\|_{ar}$, we have
$\mathcal C_j \in \mathcal L(X_{ar};\R)$ with $\|\mathcal C_j\|_{\mathcal L(X_{ar};\R)}\leq r^{-1}$.
Using $\mathcal C_j$, we have
\begin{align}\label{8}
\varepsilon_j(|z_1|^2,|z_2|^2,\mathbf v)=\mathcal C_j \circ \mathcal N(|z_1|^2,|z_2|^2,\mathbf v).
\end{align}

\begin{proposition}\label{prop:2}
There exists $r_0>0$ s.t.\ for $\delta<r\leq r_0$, we have $\mathbf \Phi \in C^\omega(B_{\R^2}(\delta^2)\times X_{ar}^c;X_{ar}^c)$.
Further, there exists $C_0>0$ s.t.\ for $|z|\leq \delta$, $\mathbf \Phi(|z_1|^2,|z_2|^2,\cdot)$ is a contraction mapping on $B_{X_{ar}^c}(Cr^7)$.
\end{proposition}

\begin{proof}
By Proposition \ref{prop:1}, Lemma \ref{lem:2.0}, \eqref{6.02}, \eqref{7} and \eqref{8}, we have $\mathbf \Phi \in C^\omega(B_{\R^2}(\delta)\times X_{ar}^c;X_{ar}^c)$.

Next, notice that since $\|\Phi_j(|z_j|^2)\|_{ar}\lesssim r$, for $\|\mathbf v\|_{ar}\leq r$, we have
\begin{align*}
\|\mathcal N(|z_1|^2,|z_2|^2,\mathbf v)\|_{ar}\leq &\|\mathcal M_7(|z_1|^2,|z_2|^2,\Phi_1(|z_1|^2)+\Phi_2(|z_2|^2)+\mathbf v)\|_{ar}\\&+\|\mathcal M_7(|z_1|^2,|z_2|^2,\Phi_1(|z_1|^2))\|_{ar}+\|\mathcal M_7(|z_1|^2,|z_2|^2,\Phi_2(|z_2|^2))\|_{ar}\lesssim r^7.
\end{align*}
Therefore, by \eqref{8}, we have $|\varepsilon(|z_1|^2,|z_2|^2,\mathbf v)|\lesssim r^6$.
Next, by Proposition \ref{prop:1}, one can show $\|\mathbf q(|z_1|^2,|z_2|^2,0)\|_{ar}\lesssim r^{13}$.
Thus, we have
\begin{align}
\|\mathbf \Phi(|z_1|^2,|z_2|^2,\mathbf v)\|_{ar}\lesssim r^7.
\end{align}
We set $C_0>0$ to satisfy $\|\mathbf \Phi(|z_1|^2,|z_2|^2,0)\|_{ar}\leq 2C_0r^7$ for all $z\in B_{\C^2}(r/2)$.

Next, by the multilinearity of $\mathcal M_7(|z_1|^2,|z_2|^2,\cdot)$, for $\mathbf v^1, \mathbf v^2 \in B_{X_{ar}^c}(r)$, we have
\begin{align*}
\|\mathcal N(|z_1|^2,|z_2|^2,\mathbf v^1)-\mathcal N(|z_1|^2,|z_2|^2,\mathbf v^2)\|_{ar}\lesssim r^6\|\mathbf v^1-\mathbf v^2\|_{ar}.
\end{align*}
Thus, we also have $|\varepsilon(|z_1|^2,|z_2|^2,\mathbf v^1)-\varepsilon(|z_1|^2,|z_2|^2,\mathbf v^2)|\lesssim r^5\|\mathbf v^1-\mathbf v^2\|_{ar}$.
Combining the above estimates, we have
\begin{align*}
\|\mathbf \Phi(|z_1|^2,|z_2|^2,\mathbf v^1)-\mathbf \Phi(|z_1|^2,|z_2|^2,\mathbf v^2)\|_{ar}\lesssim r^6 \|\mathbf v^1-\mathbf v^2 \|_{ar}.
\end{align*}
This implies that for $r\ll1$, $\mathbf \Phi(|z_1|^2,|z_2|^2,\cdot)$ is a contraction mapping on $B_{X_{ar}^c}(C_0r^7)$.
\end{proof}

Theorem \ref{thm:1} follows from Proposition \ref{prop:1} almost immediately.

\begin{proof}[Proof of Theorem \ref{thm:1}]
Let $\mathbf v(|z_1|^2,|z_2|^2)$ be the solutions of the fixed point problem \eqref{6.01}.
Then, since $\mathbf \Phi\in C^\omega(B_{\R^2}(\delta^2)\times X_{ar}^c;X_{ar}^c)$, we have $\mathbf v \in C^\omega(B_{\R^2}(\delta^2);X_{ar}^c)$.
Now, for $\mathbf v(|z_1|^2,|z_2|^2)=\{v_{jm}(|z_1|^2,|z_2|^2)\}_{j=1,2,m\geq 0}$, set
\begin{align*}
\psi(z_1,z_2)=\sum_{m\geq 0} z_1^{m+1}\overline{z_2}^mv_{1m}+ \overline{z_1}^mz_2^{m+1}  v_{2m},
\end{align*}
and $\varepsilon_j(|z_1|^2,|z_2|^2)=\varepsilon_j(|z_1|^2,|z_2|^2;\mathbf v((|z_1|^2,|z_2|^2)))$, where $\varepsilon_j$ in the r.h.s.\ is given by \eqref{8}.
Then, we have $\psi \in C^\omega(B_{\C^2}(\delta);l_e^a)$ and $\varepsilon_j \in C^\omega(B_{\R^2}(\delta^2);\R)$.
The gauge property \eqref{4.001} is a direct consequence of the form of $\psi$.
Now, since $\mathbf v$ is a unique solution of \eqref{6.01} in $B_{X_{ar}^c}(Cr^7)$, we see that $\mathbf v(|z_1|^2,0)=\mathbf v(0,|z_2|^2)=0$. Further, we have $\varepsilon (|z_1|^2,0)=\varepsilon(0,|z_2|^2)=0$.
Thus, the estimates \eqref{4.01} and \eqref{4.02} follows from Proposition \ref{prop:2} and the analyticity.
Finally the fact that $\Psi(z_1,z_2)=\phi_1(z_1)+\phi_2(z_2)+\psi(z_1,z_2)$ is a solution of \eqref{1} under the condition \eqref{4.0} follows from the construction of $\mathbf v$.
\end{proof}

\section{Coordinate}\label{sec:asymp}
In this section, we prepare the standard modulation argument for the proof of Theorem \ref{thm:2}.
We show that for $u\in l^2$ with $\|u\|_{l^2}\ll1$, there exists $z_1,z_2$ s.t.\ $u=\Psi(z_1,z_2)+v$, where $v$ corresponds to the dispersive wave.
In particular, we define a ``nonlinear continuous space" $\mathcal H_c[z_1,z_2]$ (see \eqref{9.041}) and show that we can choose $z_1,z_2$ s.t.\ $v\in \mathcal H_c[z_1,z_2]$ (Lemma \ref{lem:8}).
Further, since we want to fix the space of the dispersive wave, we introduce a map $R[z_1,z_2]:\mathcal H_c[0,0]\to \mathcal H_c[z_1,z_2]$ (Lemma \ref{lem:9}) so that we can express $u$ as
$u=\Psi(z_1,z_2)+R[z_1,z_2]\eta$ for $\eta\in \mathcal H_c[0,0]$.
As a conclusion, we obtain a coordinate $(z_1,z_2,\eta)\in \C^2\times\mathcal H_c[0,0]$ on $B_{l^2}(\delta)$ with $\delta\ll1$ (Lemma \ref{lem:10}).

We first explain how to define the ``nonlinear continuous space".
Since $\Psi(e^{-\im \E_1 t}z_1, e^{-\im \E_2 t}z_2)$ is a solution of \eqref{1} for fixed $z_1,z_2$, we have
\begin{align}\label{9}
H\Psi + |\Psi|^{6}\Psi = \im \sum_{j=1,2} \(\mathcal E_{j}z_{j,I}D_{j,R}\Psi(z_1,z_2)-\mathcal E_{j}z_{j,R}D_{j,I}\Psi(z_1,z_2)\) ,
\end{align}
where $D_{j,A} f :=\partial_{z_{j,A}}f$ for $j=1,2$, $A=R,I$.

Recall that \eqref{1} conserves the energy $E$ and the $l^2$ norm, where
\begin{align}\label{9.01}
E(u)=\frac 1 2 \<H u,  u\>+\frac 1 8 \<|u|^6u,u\>.
\end{align}
Substituting, $\Psi(z_1,z_2) + v$, we have
\begin{align}
&E(\Psi+v)=E(\Psi(z_1,z_2))+E(v)+\<H \Psi(z_1,z_2)+|\Psi|^6\Psi, v\>+N(z_1,z_2,v)\nonumber\\&=E(\Psi(z_1,z_2))+E(v)+\sum_{j=1,2}\(\mathcal E_j z_{j,I}\<\im D_{j,R}\Psi , v\>-\mathcal E_j z_{j,R}\<\im D_{j,I}\Psi , v\>\)+N(z_1,z_2,v),\label{9.02}
\end{align}
where we have used \eqref{9} in the second equality and
\begin{align}
&N(z_1,z_2,v)=\sum_{k=2}^7\sum_{i+j=k,\ i\geq j}\<G_{k,i,j}(z_1,z_2),v^i\overline{v}^j\>,\label{9.03}\\&
G_{k,i,j}(z_1,z_2)=\sum_{l+r=8-k}C_{k,i,j,l,r}\Psi(z_1,z_2)^l\overline{\Psi(z_1,z_2)}^r,\label{9.04}
\end{align}
for some $C_{k,i,j,l,r}\in \R$.
We take the orthogonality condition for $v$ to eliminate the first order term of $v$ in \eqref{9.02}.
Therefore, we set
\begin{align}\label{9.041}
\mathcal H_c[z_1,z_2]:=\{v\in l^2\ |\ \<\im v,D_{j,A}\Psi\>=0,\ j=1,2,\ A=R,I\}.
\end{align}

Now, by a standard argument using implicit function theorem, one can choose $z_1,z_2$ to have $v\in \mathcal H_c[z_1,z_2]$.
In the following, we use the notation $\phi_{j,R}:=\phi_j$ and $\phi_{j,I}=\im \phi_j$ for $j=1,2$.

%$\mathcal H_c[z_1,z_2]$ satisfies the following gauge property.
%\begin{lemma}\label{lem:7.1}
%Let $u\in \mathcal H_c[z_1,z_2]$.
%Then, for arbitrary $\theta\in \R$, we have $e^{\im \theta}u\in \mathcal H_c[e^{\im \theta}z_0,e^{\im \theta}z_1]$.
%\end{lemma}
%
%\begin{proof}
%Since $a$
%\end{proof}

\begin{lemma}\label{lem:8}
There exists $\delta>0$ s.t.\ there exists
$
(z_1(\cdot), z_2(\cdot))\in C^\omega (B_{l^2}(\delta);\C\times \C),
$
s.t.\  
\begin{align}\label{9.05}
v(u):=u-\Psi(z_1(u),z_2(u))\in \mathcal H_c[z_1(u),z_2(u)].
\end{align}
\end{lemma}

\begin{proof}
Set
\begin{align*}
\mathcal F(u,z_1,z_2):=
\begin{pmatrix}
\<\im(u-\Psi(z_1,z_2)),D_{0,R}\Psi(z_1,z_2)\>\\
\<\im(u-\Psi(z_1,z_2)),D_{0,I}\Psi(z_1,z_2)\>\\
\<\im(u-\Psi(z_1,z_2)),D_{1,R}\Psi(z_1,z_2)\>\\
\<\im(u-\Psi(z_1,z_2)),D_{1,I}\Psi(z_1,z_2)\>
\end{pmatrix}
\end{align*}
By implicit function theorem, to obtain $z_1(u),z_2(u)$ which satisfy $\mathcal F(u,z_1(u),z_2(u))=0$, it suffices to show $\left.\frac {\partial \mathcal F}{\partial (z_{1,R},z_{1,I},z_{2,R},z_{2,I})}\right|_{(u,z_1,z_2)=0}$ is invertible if $\|u\|_{l^2}\ll 1$.
Since for $j,k=1,2$, $A,B=R,I$,
$\Psi=o(1)$, $D_{j,A}\Psi=\phi_{j,A}+o(1)$ and $D_{j,A}D_{k,B}\Psi=o(1)$ as $\|u\|_{l^2}, |z_1|, |z_2|\to 0$,
we have
\begin{align*}
\left.\frac {\partial \mathcal F}{\partial (z_{1,R},z_{1,I},z_{2,R},z_{2,I})}\right|_{(u,z_1,z_2)=(0,0,0)}=\begin{pmatrix} 0 & 1 & 0 & 0\\
-1 & 0  & 0 & 0 \\ 0 & 0 & 0 & 1 \\ 0 & 0 & -1 & 0 \end{pmatrix}.
\end{align*}
Therefore, there exists $(z_1(\cdot), z_2(\cdot))\in C^\omega (B_{l^2}(\delta);\C\times \C)$ s.t.\ $\mathcal F(u,z_1(u),z_2(u))=0$ which is equivalent to \eqref{9.05}.
%Finally, if $u-\Psi(z_1(u),z_2(u))\in \mathcal H_c[z_1(u),z_2(u)]$, by Lemma \ref{lem:7.1} and \eqref{4.001},
%we have, 
%\begin{align*}
%e^{\im \theta}u-\Psi(e^{\im \theta}z_0(u),e^{\im \theta}z_1(u))=e^{\im \theta}u-e^{\im \theta}\Psi(z_1(u),z_2(u))\in \mathcal H_c[e^{\im \theta}z_0(u),e^{\im \theta}z_1(u)].
%\end{align*}
%Therefore, by the uniqueness of the solution of $\mathcal F(u,z_1,z_2)=0$, we have $z_j(e^{\im \theta}u)=e^{\im \theta}z_j(u)$ for $j=1,2$ and $v(e^{\im \theta}u)=e^{\im \theta}v(u)$ also follows from \eqref{4.001}.
\end{proof}

Next, set
\begin{align*}
&P_d:=\sum_{j=1,2,A=R,I}\<\cdot,\phi_{j,A}\>\phi_{j,A},\quad
P_c:=1-P_d.
\end{align*}
Notice that $\mathcal H_c[0,0]=P_c l^2=:l_c^2$.

We now define the inverse of the map $\left.P_c\right|_{\mathcal H_c[z_1,z_2]}$ which was used in \cite{GNT04}.

\begin{lemma}\label{lem:9}
There exists $\delta>0$ s.t. there exists $\alpha_{j,A}\in C^\omega (B_{\C^2}(\delta);l_e^{a_1}(\Z;\C))$ ($j=1,2$, $A=R,I$), where, $a_1$ is the constant given in Theorem \ref{thm:1},
 s.t.\ 
%for arbitrary multi-index $\mathbf m=(m_1,\cdots,m_{|\mathbf m|})$, $m_k\in \{(j,A)\ |\ j=1,2,A=R,I\}$, $|\mathbf m|\geq 0$,
\begin{align*}
\|\alpha_{j,A}(z_1,z_2)\|_{l_e^{a_1}}\lesssim|z|^{6},
\end{align*}
Further, 
\begin{align}\label{9.2}
R[z_1,z_2]\eta=\eta + \sum_{j=1,2,A=R,I}\<\alpha_{j,A}(z_1,z_2), \eta\>\phi_{j,A}.
\end{align}
satisfies $R[z_1,z_2]:l_c^2\to \mathcal H_c[z_1,z_2]$ and $\left.P_c\right|_{\mathcal H_c[z_1,z_2]}=R[z_1,z_2]^{-1}$.
%Further, we have
%\begin{align}
%&\alpha_{j,A}(e^{\im\theta}z_0,e^{\im\theta}z_1)=e^{\im \theta}\alpha_{j,A}(z_1,z_2),\label{9.3}\\&
%R[e^{\im\theta}z_0,e^{\im\theta}z_1]=e^{\im \theta}R[z_1,z_2]e^{-\im \theta}.\label{9.4}
%\end{align}
\end{lemma}

\begin{proof}
We define $\beta_{j,A}[z_1,z_2]\eta\in \R$ ($j=1,2$, $A=R,I$) for $\eta\in l_c^2$ to be the unique solution of
\begin{align}\label{10}
\<\im\(\eta + \sum_{j=1,2, A=R,I}(\beta_{j,A}[z_1,z_2]\eta)\phi_{j,A}\),D_{k,B}\Psi(z_1,z_2)\>=0,
\end{align}
for $k=0,1$, $B=R,I$ and set
\begin{align*}
R[z_1,z_2]\eta=\eta + \sum_{j=1,2,A=R,I}\(\beta_{j,A}(z_1,z_2)\eta\)\phi_{j,A}
\end{align*}
By the form of $R[z_1,z_2]$, it is obvious that $P_cR[z_1,z_2]=\mathrm{id}_{l_c^2}$.
On the other hand for $\eta\in \mathcal H_c[z_1,z_2]$, we have
\begin{align*}
R[z_1,z_2]P_c\eta&=P_c\eta +  \sum_{j=1,2,A=R,I}(\beta_{j,A}(z_1,z_2)P_c\eta)\phi_{j,A})\\&=
\eta +\sum_{j=1,2,A=R,I}\((\beta_{j,A}(z_1,z_2)P_c\eta)-\<\eta,\phi_{j,A}\>\)\phi_{j,A}.
\end{align*}
Since $P_c\eta + \sum_{j=1,2,A=R,I}\<\eta,\phi_{j,A}\>\phi_{j,A} \in \mathcal H_c[z_1,z_2]$, by the uniqueness of the solution of \eqref{10}, we have
\begin{align*}
\beta_{j,A}(z_1,z_2)P_c\eta=\<\eta,\phi_{j,A}\>,\quad j=1,2,\ A=R,I.
\end{align*}
Therefore, we have $R[z_1,z_2]P_c\eta=\eta$.

We finally prove \eqref{10} has a unique solution.
\eqref{10} can be written as
\begin{align}\label{10.01}
\sum_{j=1,2,A=R,I}\(\beta_{j,A}(z_1,z_2)\eta\< \phi_{j,A},D_{k,B}\Psi\>\)=-\<\im \eta, D_{k,B}\Psi\>=-\<\im \eta, D_{k,B}(q_0+q_1+\psi)\>,
\end{align}
where $k=0,1$ and $B=R,I$.
Writing \eqref{10.01} in the matrix form, one can see the coefficient matrix becomes invertible.
Therefore, we have a unique solution of \eqref{10} and the solution $\beta_{j,A}[z_1,z_2]\eta$ can be expressed as $\<\alpha_{j,A}(z_1,z_2),\eta\>$ where $\alpha_{j,A}(z_1,z_2)$ are linear combinations of $D_{k,B}(q_0+q_1+\psi)$ for $k=0,1$ and $B=R,I$.
This expression combined with Theorem \ref{thm:1} gives us the desired estimates for $\alpha_{j,A}$ for $j=1,2$ and $A=R,I$.
\end{proof}

Combining Lemmas \ref{lem:8}, \ref{lem:9}, we obtain a system of coordinates near the origin of $l^2$.
\begin{lemma}\label{lem:10}
Let $\delta>0$ sufficiently small.
Then there exits a $C^\omega$ diffeomorphism
\begin{align}\label{10.1}
B_{\C^2\times l_c^2}(\delta)\ni (z_1,z_2,\eta)\mapsto u=\Psi(z_1,z_2)+R[z_1,z_2]\eta \in l^2.
\end{align}
Further, we have
\begin{align}\label{11}
|z_1|+|z_2|+\|\eta\|_{l^2}\sim \|u\|_{l^2}.
\end{align}
%and the gauge properties
%\begin{align*}
%&u(e^{\im \theta}z_0,e^{\im \theta}z_1,e^{\im \theta}\eta)=e^{\im \theta}u(z_1,z_2,\eta).\\&
%z_0(e^{\im \theta}u)=e^{\im \theta}z_0(u),
%z_1(e^{\im \theta}u)=e^{\im \theta}z_1(u),
%\eta(e^{\im \theta}u)=e^{\im \theta}\eta(u),
%\end{align*}
%Let $\|u\|_{l^2}\ll 1$.
%Then, there exists $z_1(u),z_2(u)\in \C$ and $\eta(u)\in \mathcal H_c[0,0]$ s.t.
%\begin{align*}
%u=\Psi(z_1(u),z_2(u))+R[z_1(u),z_2(u)]\eta.
%\end{align*}
\end{lemma}

In the following, we set $(z_1(u),z_2(u),\eta(u))\in \C\times \C\times  l_c^2$ to be the inverse of the map \eqref{10.1}.
%\begin{proof}
%
%\end{proof}

\section{Darboux theorem}\label{sec:darboux}
In the previous section, we have introduced a coordinate $(z_1,z_2,\eta)$ to express small $l^2$ functions.
Since, $l^2$ is conserved by the flow of \eqref{1}, we can study the dynamics of small solutions of \eqref{1} in this coordinate.
Indeed, since we have the equation \eqref{1} and four orthogonal conditions \eqref{9.041}, \eqref{1} becomes a system of one discrete evolution equation and four ODEs.
However, due to the complexity of the quasi-periodic solution itself, it seems to be difficult to handle this system directly.
Therefore, following \cite{CuMaAPDE}, we make a change of coordinate to have a "canonical" coordinate system and moreover have a simple system of equations \eqref{40}-\eqref{41}, given in the end of this section.

In the following, we introduce exterior derivatives and symplectic forms.

\begin{definition}[Exterior derivative]
Let $F \in C^\infty (l^2;\mathbf B)$, where $\mathbf B$ be a Banach space (in particular we are considering the case $\mathbf B=\R,\C,l^2$).
We think $F$ as a $0$-form and define its exterior derivative $dF(u)$ (which is a $1$-form) by
$dF(u)=D F(u)$, where $D F(u)$ is the Fr\'echet derivative of $F$.
Next, let $\omega(u)$ be $1$-form.
Then, we define its exterior derivative $d \omega(u)$ (which is a $2$-form) by
\begin{align}\label{14}
d \omega (u) (X,Y)=\mathcal L_X \omega(u) (Y) - \mathcal L_Y \omega(u) (X),
\end{align}
where $\mathcal L_X$ is the Lie derivative (i.e.\ $\mathcal L_X \omega(u) (Y)=\left.\frac{d}{d \varepsilon}\right|_{\varepsilon=0} \omega(u+ \varepsilon X)(Y)$).
\end{definition}

\begin{remark}
In general, for the definition of the exterior derivative, we have to add  $-\omega(\mathcal L_X(Y))$ to \eqref{14}.
However, our space $l^2$ is flat and we only have to consider constant vector fields for the definition, we can define $d\omega$ as \eqref{14}.
See section 6.4 of \cite{AMRBook}.

\end{remark}

We set the symplectic form $\Omega$ associated to \eqref{1} by
\begin{align*}
\Omega(X,Y):=\<\im X ,Y \>,
\end{align*}
and
\begin{align*}
B(u)X:=\frac 1 2 \Omega(u,X).
\end{align*}
Then,
\begin{align}\label{15}
dB(u)(X,Y)=\mathcal L_XB(u)Y-\mathcal L_Y B(u)Y=\frac 1 2 \Omega (X,Y)-\frac 1 2 \Omega(Y,X)=\Omega(X,Y).
\end{align}
Therefore, we have $dB(u)=\Omega$.

Next, we introduce a new symplectic form $\Omega_0$.

\begin{definition}
We define the $1$-form $B_0$ and $2$-form $\Omega_0$ by the following.
\begin{align*}
&B_0(u)X:=\frac 1 2 \Omega(\Psi(z_1,z_2), d\Psi(z_1,z_2)(X))+\frac 1 2 \Omega(\eta,d\eta(X)),\\&
\Omega_0(X,Y)=\Omega(d\Psi(z_1,z_2)(X),d\Psi(z_1,z_2)(Y))+\Omega(d\eta(X),d\eta(Y)).
\end{align*}
\end{definition}

\begin{remark}
As \eqref{15}, we see
$
dB_0(u)=\Omega_0.
$
\end{remark}

\begin{remark}
The original symplectic form $\Omega$ do not depend on $u$.
However, the new symplectic form $\Omega_0$ depends on $u$.
So, $\Omega_0(X,Y)$ should be written as $\Omega_0(u)(X,Y)$.
However, we omit $u$ since there should be no confusion.
\end{remark}

Our aim is to change the coordinate system $(z_1,z_2,\eta)$ to have the new symplectic form $\Omega_0$.
To do so, we use the Moser's argument.
Let $\Gamma$ s.t.\ $\Omega-\Omega_0=d \Gamma$ and $\mathcal X^s$ satisfies $i_{\mathcal X^s}(\Omega_0+s(\Omega-\Omega_0))=-\Gamma$, where $i_X \omega(Y)=\omega(X,Y)$.
Then, if we set $\mathcal Y_s$ to be the solution map of $\frac{d}{ds}\mathcal Y_s =\mathcal X^s(\mathcal Y_s)$, we have
\begin{align}\label{15.01}
\frac{d}{d s}(\mathcal Y_s^* \Omega_s)=\mathcal Y_s^*(\mathcal L_{\mathcal X^s}\Omega_s + \partial \Omega_s)=\mathcal Y_s^*(di_{\mathcal X^s}\Omega_s + d \Gamma)=0,
\end{align}
where $\Omega_s=\Omega_0+s(\Omega-\Omega_0)$.
Thus, we have the desired change of coordinate $\mathcal Y= \mathcal Y_1$ which satisfies $\mathcal Y^*\Omega=\Omega_0$.
By this argument, it may look like we have already have the change of the coordinate.
However, for the application to the asymptotic stability of the quasi-periodic solution, we need have an estimate of $\mathcal Y$ in some weighted space.

In the following we construct $\Gamma$ and $\mathcal X^s$ directly.

\begin{lemma}\label{lem:11}
Let $\delta>0$ sufficiently small.
Then, there exists $F_\eta\in C^\omega(B_{\C^2\times P_cl_e^{-a_1}}(\delta);l_e^{a_1})$ and $F_{j,A}\in C^\omega(B_{\C^2\times P_cl_e^{-a_1}}(\delta);l_e^{a_1})$ $(j=1,2,\ A=R,I)$ s.t. there exists $C$ s.t.
\begin{align*}
B(u)-B_0(u)-dC=\sum_{j=1,2,A=R,I}\<F_{j,A},\eta\>dz_{j,A}+\<F_\eta,d\eta\>=:\Gamma.
\end{align*}
Further, for $j=1,2,A=R,I$, we have%$\mathbf m=(m_1,\cdots,m_{|\mathbf m|})$ with $m_j\in \{(k,B)\ |\ k=0,1,B=R,I\}$,
\begin{align}
\|F_\eta\|_{l_e^{a_1}}\lesssim  |z|^{6}\|\eta\|_{l_e^{-a_1}},\quad
\|F_{j,A}\|_{l_e^{a_1}}\lesssim |z|^{6}.\label{15.1}
\end{align}
%where $\nabla_\eta$ is the Frech\'et derivative with respect to $\eta$.
\end{lemma}

\begin{proof}
In the following, we write $\Sigma_{j=1,2,A=R,I}$ as $\Sigma_{j,A}$. Further $\Sigma_{k,B}$ and $\Sigma_{l,C}$ will have the same meaning.
First, since
\begin{align*}
2B(u)=&\Omega(u,du)\\
=&\Omega(\Psi + \eta + \sum_{j,A}\<\alpha_{j,A},\eta\>\phi_{j,A}, d(\Psi + \eta + \sum_{k,B}\<\alpha_{k,B},\eta\>\phi_{k,B}))\\=&
\Omega(\Psi,d\Psi)+\Omega(\eta,d\eta) + \Omega(\Psi, d\eta)+ \Omega(\eta, d\Psi)\\&
\quad + \sum_{k,B}\Omega(\Psi, \phi_{k,B})d(\<\alpha_{k,B},\eta\>)+\sum_{j,A}\sum_{k,B}\Omega(\phi_{j,A},\phi_{k,B})\<\alpha_{a,A},\eta\> d(\<\alpha_{k,B},\eta\>).
\end{align*}
So, we have
\begin{align}
2(B(u)-B_0(u))=&\Omega(\Psi, d\eta)+ \Omega(\eta, d\Psi)\label{16}
 + \sum_{k,B}\Omega(\Psi, \phi_{k,B})d(\<\alpha_{k,B},\eta\>)\\&+\sum_{j,A}\sum_{k,B}\Omega(\phi_{j,A},\phi_{k,B})\<\alpha_{a,A},\eta\> d(\<\alpha_{k,B},\eta\>).\nonumber
\end{align}
The first and second term of r.h.s.\ of \eqref{16} can be rewritten as
\begin{align}\label{17}
\Omega(\Psi, d\eta)+ \Omega(\eta, d\Psi)=d \Omega(\Psi,\eta) + 2 \Omega(\eta, d\Psi).
\end{align}
The third term of r.h.s.\ of \eqref{16} can be rewritten as
\begin{align}\label{18}
\sum_{k,B}\Omega(\Psi, \phi_{k,B})d(\<\alpha_{k,B},\eta\>)=d\(\sum_{k,B}\Omega(\Psi,\phi_{k,B})\<\alpha_{k,B},\eta\>\)+\sum_{k,B}\<\alpha_{k,B},\eta\>\Omega(\phi_{k,B},d\Psi).
\end{align}
The last term of \eqref{16} can be rewritten as
\begin{align}\label{19}
&\sum_{j,A}\sum_{k,B}\Omega(\phi_{j,A},\phi_{k,B})\<\alpha_{a,A},\eta\> d(\<\alpha_{k,B},\eta\>)\nonumber\\&\quad=\sum_{j,A}\sum_{k,B}\Omega(\phi_{j,A},\phi_{k,B})\<\alpha_{a,A},\eta\>\( \<\eta,d\alpha_{k,B}\>+\<\alpha_{k,B},d\eta\>\).
\end{align}
Combining \eqref{17}, \eqref{18} and \eqref{19}, we have
\begin{align}\label{20}
&2(B(u)-B_0(u))=d\(\Omega(\Psi,\eta)+\sum_{k,B}\Omega(\Psi,\phi_{k,B})\<\alpha_{k,B},\eta\>\)\\&
\quad+\Omega\( 2\eta + \sum_{k,B}\<\alpha_{k,B},\eta\>\phi_{k,B}, d\Psi\)+\sum_{l,C}\sum_{k,B}\Omega(\phi_{l,C},\phi_{k,B})\<\alpha_{l,C},\eta\>\( \<\eta,d\alpha_{k,B}\>+\<\alpha_{k,B},d\eta\>\).\nonumber
\end{align}
Since $d\Psi=\sum_{j,A}D_{j,A}\Psi dz_{j,A}$, $d \alpha_{k,B}=\sum_{j,A}D_{j,A}\alpha_{k,B}dz_{j,A}$ and $\Omega(\eta, D_{j,A}\Psi)=\Omega(\eta, D_{j,A}(q_0+q_1+\psi))$ because $P_c\eta=\eta$, from \eqref{20}, we have
\begin{align}
B(u)-B_0(u)= d C + \< F_{j,A}, \eta\>dz_{j,A}+\<F,d\eta\>,
\end{align}
where
\begin{align}
C=&\frac 1 2 \(\Omega(\Psi,\eta)+\sum_{k,B}\Omega(\Psi,\phi_{k,B})\<\alpha_{k,B},\eta\>\),\nonumber\\
F_{j,A}=&-\im D_{j,A}(q_0+q_1+\psi)+\frac 1 2 \sum_{k,B}\Omega(\phi_{k,B},D_{j,A}\Psi)\alpha_{k,B}\label{21}\\&+\frac 1 2 \sum_{k,B}\sum_{l,C}\Omega(\phi_{k,B},\phi_{l,C})\<\eta,D_{j,A}\alpha_{l,C}\>\alpha_{k,B},\nonumber\\
F_\eta=&\sum_{k,B}\sum_{l,C}\Omega(\phi_{k,B},\phi_{l,C})\<\alpha_{k,B},\eta\>\alpha_{l,C}\label{22}
\end{align}
The estimates of \eqref{15.1} follows from \eqref{21} and \eqref{22} and Lemma \ref{lem:9}.
\end{proof}

By lemma \ref{11}, we have
\begin{align*}
\Omega-\Omega_0=d(B(u)-B_0(u))=d(dC+\Gamma)=d \Gamma.
\end{align*}
We set
\begin{align*}
\Omega_s=\Omega_0+s (\Omega-\Omega_0),
\end{align*}
and try to find a solution $\mathcal X^s$ of the equation $i_{\mathcal X^s}\Omega_s=-\Gamma$.
\begin{lemma}\label{lem:12}
Let $\delta>0$ sufficiently small.
Then, there exist $\mathcal X^s_\eta\in C^\omega(B_{\C^2\times P_cl_e^{-a_1}}(\delta);l_e^{a_1})$ and $\mathcal X^s_{j,A}\in C^\omega(B_{\C^2\times P_cl_e^{-a_1}}(\delta);\R)$ for $j=1,2$, $A=R,I$ s.t.
$\mathcal X^s := \sum_{j,A}\mathcal X^s_{j,A}\partial_{z_{j,A}}+\mathcal X^s_\eta \nabla_\eta$
satisfies $i_{\mathcal X^s}\Omega_s=-\Gamma$.
Further, we have
\begin{align*}
&\|\mathcal X^s_\eta\|_{l_e^{a_1}}+\sum_{j,A}|\mathcal X^s_{j,A}|\lesssim  |z|^{6}\|\eta\|_{l_e^{-a_1}}.
\end{align*}
%where $\mathbf m=(m_1,\cdots,m_{|\mathbf m|})$ with $m_j\in \{(k,B)\ |\ k=0,1,B=R,I\}$.
\end{lemma}

\begin{proof}
We directly solve
\begin{align}\label{24}
\Omega_0(\mathcal X^s,\cdot)+s\(\Omega(\mathcal X^s,\cdot)-\Omega_0(\mathcal X^s,\cdot)\)=-\Gamma.
\end{align}
%where $\mathcal X^s := \sum_{j,A}\mathcal X^s_{j,A}\partial_{z_{j,A}}+\mathcal X^s_\eta \nabla_\eta$
%(In other words, we set $\mathcal X^s_{j,A}=dz_{j,A} \mathcal X^s$ and $\mathcal X^s_\eta = d\eta \mathcal X^s$).
In the following, we omit the summation over $j=1,2,A=R,I$, etc.\ and
$j,k,l,r$ will always be $1,2$ and $A,B,C,D$ will be $R,I$.
First, we have
\begin{align*}
\Omega_0(\mathcal X^s,\cdot)&=\Omega(d\Psi(\mathcal X^s),d\Psi)+\Omega(\mathcal X^s_\eta,d\eta)\\&=
\Omega(D_{k,B}\Psi,D_{j,A}\Psi)\mathcal X^s_{k,B}dz_{j,A}+\Omega(\mathcal X^s_\eta,d\eta),
\end{align*}
and
\begin{align*}
\Omega(\mathcal X^s,\cdot)&=\Omega(d(\Psi+\eta+\<\alpha_{l,C},\eta\>\phi_{l,C})(\mathcal X^s),d(\Psi+\eta+\<\alpha_{r,D},\eta\>\phi_{r,D}))\\&
=\Omega(D_{k,B}\Psi \mathcal X^s_{k,B}+\mathcal X^s_\eta+\<D_{k,B}\alpha_{l,C},\eta\>)\phi_{l,C}\mathcal X^s_{k,B}+\<\alpha_{l,C},\mathcal X^s_\eta\>\phi_{l,C},\\&\quad\quad
D_{j,A}\Psi dz_{j,A}+d\eta+\<D_{j,A}\alpha_{r,D},\eta\>\phi_{r,D}dz_{j,A}+\<\alpha_{r,D},d\eta\>\phi_{r,D})\\&=
\Omega_0(\mathcal X^s,\cdot)+\(G_{j,A,k,B}\mathcal X^s_{k,B}+\<G_{j,A,\eta},\mathcal X^s_\eta\>\)dz_{j,A}-\mathcal X^s_{k,B}\<G_{k,B,\eta},d\eta\>\\&\quad+\Omega(\phi_{l,C},\phi_{r,D})\<\alpha_{l,C},\mathcal X^s_\eta\>\<\alpha_{r,D},d\eta\>
\end{align*}
where
\begin{align*}
G_{j,A,k,B}=&\Omega(D_{k,B}\Psi,\phi_{r,D})\<D_{j,A}\alpha_{r,D},\eta\>+\<D_{k,B}\alpha_{l,C},\eta\>\Omega(\phi_{l,C},D_{j,A}\Psi)\\&+\<D_{k,B}\alpha_{l,C},\eta\>\Omega(\phi_{l,C},\phi_{r,D})\<D_{j,A}\alpha_{r,D},\eta\>,\\
G_{j,A,\eta}=&-\im D_{j,A}(q_0+q_1+\psi)+\Omega(\phi_{l,C},D_{j,A}\Psi)\alpha_{l,C}+\Omega(\phi_{l,C},\phi_{r,D})\<D_{j,A}\alpha_{r,D},\eta\>\alpha_{l,C}.
%\\
%G_{\eta,k,B}=&\im D_{k,B}(q_0+q_1+\psi)+\Omega(D_{k,B}\Psi,\phi_{r,D})\alpha_{r,D}+\<D_{k,B}\alpha_{l,C},\eta\>\Omega(\phi_{l,C},\phi_{r,D})\alpha_{r,D}.
\end{align*}
Therefore, \eqref{24} can be written as
\begin{align}
&\(\Omega(D_{k,B}\Psi,D_{j,A}\Psi)+sG_{j,A,k,B}\)\mathcal X^s_{k,B}+\<G_{j,A,\eta},\mathcal X^s_\eta\>=-\<F_{j,A},\eta\>\label{25}\\&
\im \mathcal X^s_\eta + s\(-\mathcal X^s_{k,B}G_{k,B,\eta}+\Omega(\phi_{l,C},\phi_{r,D})\<\alpha_{l,C},\mathcal X^s_\eta\>\alpha_{r,D}\)=-F_\eta,\label{26}
\end{align}
where $F_\eta$ and $F_{j,A}$ are given in Lemma \ref{lem:11}.
We first solve \eqref{26} fixing $|\mathcal X^s_{k,B}|\leq 1$.
Notice that \eqref{26} can be rewritten as
\begin{align*}
(1-\mathcal A)\mathcal X^s_\eta=-\im s \mathcal X^s_{k,B} G_{k,B,\eta}+\im F_\eta,
\end{align*}
where
\begin{align*}
\mathcal A\xi=\im s \Omega(\phi_{l,C},\phi_{r,D})\<\alpha_{l,C},\xi\>\alpha_{r,D}.
\end{align*}
Since $\|\mathcal A\|_{l_e^{a_1}\to l_e^{a_1}}\lesssim |z|^6$, we have $\|(1-\mathcal A)^{-1}\|_{\mathcal L(l_e^{a_1})}\lesssim 1$ by Lemma \ref{lem:9}.
Therefore, we have
\begin{align}\label{27}
\mathcal X^s_\eta=(1-\mathcal A)^{-1}\(-\im s \mathcal X^s_{k,B} G_{k,B,\eta}+\im F_\eta\).
\end{align}
Substituting, \eqref{27} into \eqref{25}, we have
\begin{align}\label{28}
&\(\Omega(D_{k,B}\Psi,D_{j,A}\Psi)+s G_{j,A,k,B}-\<G_{j,A,\eta},\im s(1-\mathcal A)^{-1})G_{k,B,\eta}\>\)\mathcal X^s_{k,B}\nonumber\\&\quad=-\<F_{j,A},\eta\>-\<G_{j,A,\eta},\im(1-\mathcal A)^{-1}F_\eta\>.
\end{align}
Considering $\Omega(D_{k,B}\Psi,D_{j,A}\Psi)+s G_{j,A,k,B}+\<G_{j,A,\eta},\im s(1-\mathcal A)^{-1})G_{\eta,k,B}\>$ as a $4\times 4$ matrix, this matrix has the form
\begin{align*}
\begin{pmatrix}
0 & -1 & 0 & 0 \\
1 & 0 & 0 & 0 \\
0 & 0 & 0 & -1 \\
0 & 0 & 1 & 0 
\end{pmatrix}+o(1).
\end{align*}
Since this matrix in invertible, we have the solution of \eqref{28}.
Therefore, we have the solution of \eqref{24}.
Further, we have
\begin{align*}
|\mathcal X^s_{k,B}|\lesssim \|F_{j,A}\|_{l_e^{a_1}}\|\eta\|_{l_e^{-a_1}}+\|G_{j,A,\eta}\|_{l^2}\|F\|_{l^2}\lesssim |z|^6 \|\eta\|_{l_e^{-a_1}},
\end{align*}
and
\begin{align*}
\|\mathcal X^s_\eta\|_{l_e^{a_1}}\lesssim |\mathcal X^s_{k,B}| \|G_{\eta,k,B}\|_{l_e^{a_1}} +\|F\|_{l_e^{a_1}}\lesssim |z|^6 \|\eta\|_{l_e^{-a_1}}.
\end{align*}
\end{proof}

We now construct the desired change of coordinate $\mathcal Y$ by the flow of $\mathcal X^s$.
%Set $\mathcal X^s_j=\mathcal X^s_{j,R}+\im \mathcal X^s_{j,I}$ for $j=1,2$.
We consider the following system
\begin{align}\label{29}
&\frac{d}{ds} r_{j}(z_1,z_2,\eta;s)=\mathcal X^s_{j,A}(z_1+r_1(z_1,z_2,\eta;s),z_2+r_2(z_1,z_2,\eta;s), \eta+r_\eta(z_1,z_2,\eta;s)),\\&
\frac{d}{d s}r_\eta(z_1,z_2,\eta;s) =\mathcal X^s_\eta(z_1+r_1(z_1,z_2,\eta;s),z_2+r_2(z_1,z_2,\eta;s),\eta+ r_\eta(z_1,z_2,\eta;s)),\label{30}
\end{align}
with $j=1,2$ and the initial condition $r_{1}(0)=0$, $r_{2}(0)=0$ and $r_\eta(0)=0$.
\begin{lemma}\label{lem:13}
Let $\delta>0$ sufficiently small.
Then, there exists $$(r_1,r_2,r_\eta)\in C^\omega(B_{\C^2\times P_cl_e^{-a_1}}(\delta);C([0,1];\C^2\times l_c^2)),$$ s.t.
$(r_1(z_1,z_2,\eta;\cdot),r_2(z_1,z_2,\eta;\cdot),r_\eta(z_1,z_2,\eta;\cdot))$ is the solution of system \eqref{29}--\eqref{30} and
\begin{align}
&\sum_{j=1,2}|r_{j}(z_1,z_2,\eta;1)|+
\|r_\eta(z_1,z_2,\eta;1)\|_{l_e^{a_1}}\lesssim |z|^{6}\|\eta\|_{l_e^{-a_1}}.\label{30.1}
\end{align}
%where $\mathbf m=(m_1,\cdots,m_{|\mathbf m|})$ with $m_j\in \{(k,B)\ |\ k=0,1,B=R,I\}$.

\end{lemma}

\begin{proof}
We solve the system \eqref{29}--\eqref{30} by implicit function theorem.
Let $\delta>0$ sufficiently small.
Let $(x_1,x_2,\xi)\in C([0,1];B_{\C^2\times l_c^2}(\delta))$ and set
\begin{align*}
&\Phi(z_1,z_2,\eta,x_1,x_2, \xi)(s)\\&\quad=\(\Phi_0(z_1,z_2,\eta,x_1,x_2, \xi)(s),\Phi_1(z_1,z_2,\eta,x_1,x_2, \xi)(s),\Phi_\eta(z_1,z_2,\eta,x_1,x_2, \xi)(s)\),
%\(\int_0^s\mathcal X^\tau_{0}(r_{k,B}(s)+z_{k,B}, r_\eta^s+\eta)\,ds,\int_0^s\mathcal X^s_\eta(r_{k,B}(s)+z_{k,B}, r_\eta^s+\eta)\,ds\).
\end{align*}
where
\begin{align*}
&\Phi_j(z_1,z_2,\eta,x_1,x_2,\xi)(s)=x_j(s)-\int_0^s\mathcal X_j^{\tau}(z_1+x_1(\tau),z_2+x_2(\tau),\eta+\xi(\tau))\,d\tau,\quad j=1,2,\\&
\Phi_\eta(z_1,z_2,\eta,x_1,x_2,\xi)(s)=\xi(s)-\int_0^s\mathcal X_\eta^\tau(z_1+x_1(\tau),z_2+x_2(\tau),\eta+\xi(\tau))\,d\tau.
\end{align*}
Notice that $\Phi\in C^\omega(B_{\C^2\times P_cl_e^{-a_1}}(\delta)\times B_{C([0,1];\C^2\times P_cl_e^{-a_1})}(\delta);C([0,1];\C^2\times l_c^2))$.
By the estimate of lemma \ref{lem:12} and the analyticity of $\mathcal X_\eta^s$, we have
\begin{align*}
D_\xi \Phi_\eta(0,0,0,0,0,0)=\mathrm{id}_{C([0,1];l_c^2)}.
\end{align*}
Therefore, there exists $\tilde r_\eta(z_1,z_2,\eta,x_1,x_2)$ s.t.\ $\Phi_\eta(z_1,z_2,\eta,x_1,x_2,r_\eta(z_1,z_2,\eta,x_1,x_2))=0$.
Repeatedly, we will have $\tilde r_2(z_1,z_2,\eta,x_1)$ and $r_1(z_1,z_2,\eta)$ with desired property.
Therefore, setting $r_2=\tilde r_2(z_1,z_2,\eta,r_1(z_1,z_2,\eta))$ and $r_\eta(z_1,z_2,\eta)=\tilde r_\eta(z_1,z_2,\eta,r_1(z_1,z_2,\eta),\tilde r_2(z_1,z_2,\eta))$, we have the solution of \eqref{29}--\eqref{30}.
We now prove \eqref{30.1}.
From
\begin{align}
&r_j(z_1,z_2,\eta;s)=\int_0^s\mathcal X_j^{\tau}(z_1+r_1(z_1,z_2,\eta;\tau),z_2+r_2(z_1,z_2,\eta;\tau),\eta+r_\eta(z_1,z_2,\eta;\tau))\,d\tau,\label{30.3}\\&
r_\eta(z_1,z_2,\eta;s)=\int_0^s\mathcal X_\eta^{\tau}(z_1+r_1(z_1,z_2,\eta;\tau),z_2+r_2(z_1,z_2,\eta;\tau),\eta+r_\eta(z_1,z_2,\eta;\tau))\,d\tau,\label{30.4}
\end{align}
and Lemma \ref{lem:12}, setting $A(s):=\sup_{\tau\in[0,s]}\(|r_1(\tau)|+|r_2(\tau)|+\|r_\eta(\tau)\|_{l_e^{a_1}}\)$, we have
\begin{align}\label{30.5}
A(s)\leq \(|z|^6+A(s)^6\)(\|\eta\|_{l_e^{-a_1}}+A(s)).
\end{align}
Combining \eqref{30.5} with the fact $A(0)=0$, we have \eqref{30.1}.% for $|\mathbf m|=0$.
%We next differentiate \eqref{30.3}--\eqref{30.4} by $\eta$.
%Recall that since 
%\begin{align*}
%&\nabla_\eta\( F(\eta+G(\eta))\)X=\left.\frac{d}{d\epsilon}\right|_{\epsilon=0}F(\eta+\epsilon X+G(\eta+\epsilon X))=
%\left.\frac{d}{d \epsilon}\right|_{\epsilon=0}F(\eta+G(\eta)+\epsilon(X+\nabla_\eta G(\eta)X))\\&
%=\<\nabla_\eta F(\eta+G(\eta)),X+\nabla_\eta G(\eta)X\>=\<(1+(\nabla_\eta G(\eta))^*)\nabla_\eta F(\eta+G(\eta)),X\>,
%\end{align*}
%we have $\nabla_\eta \(F(\eta+G(\eta))\)=(1+(\nabla_\eta G(\eta))^*)\nabla_\eta F(\eta+G(\eta))$.
%Thus, differentiating \eqref{30.3} by $\eta$, we have
%\begin{align*}
%\nabla_\eta r_j(s)=\int_0^s D_k\mathcal X^\tau_j \nabla_\eta r_k(\tau)+(1+(\nabla_\eta r_\eta)^*)\nabla_\eta \mathcal X_j^\tau,\\
%\nabla_\eta r_\eta(s)(X)=\int_0^s\<\nabla_\eta r_k, X\>D_k\mathcal X_\eta^\tau +\nabla_\eta \mathcal X_\eta^\tau (\nabla_\eta r_\eta(X)).
%\end{align*}
%Again, setting $B(s)=\sup_{s\in[0,s]}\|\nabla r_0\|_{l_e^{a_1}}+\|\nabla_\eta r_1\|_{l_e^{a_1}}+\|\nabla_\eta r_\eta\|_{l_e^{-a_1}\to l_e^{a_1}}$, we have
%\begin{align*}
%B(s)\leq (|z_0|^5+|z_1|^5)B(s)+(|z_0|^6+|z_1|^6)
%\end{align*}
%Therefore, we have \eqref{30.2} with $|\mathbf m|=0$.
%The estimates for \eqref{30.21} with $|\mathbf m|=0$ and \eqref{30.1}--\eqref{30.21} with $|\mathbf m|\geq 1$ follows from similar argument and we have the conclusion.
\end{proof}

Now, define $\mathcal Y_s$ by 
\begin{align*}
\mathcal Y_s^*z_j=z_{j}+r_{j}(z_1,z_2,\eta;s), \quad j=1,2,\quad
\mathcal Y_s^*\eta = \eta+r_\eta(z_1,z_2,\eta;s).
\end{align*}
Then, $\mathcal Y_s$ satisfies
\begin{align*}
\frac {d}{ds}\mathcal Y_s = \mathcal X^s(\mathcal Y_s),
\end{align*}
which gives us the desired coordinate change by \eqref{15.01}.
Therefore, setting $\mathcal Y:=\mathcal Y_1$, 
we have
\begin{align*}
\mathcal Y^* \Omega=\Omega_0.
\end{align*}

We set $r_j(z_1,z_2,\eta):=r_j(z_1,z_2,\eta;1)$ for $j=1,2$ and $r_\eta(z_1,z_2,\eta):=r_\eta(z_1,z_2,\eta;1)$.

\begin{remark}\label{rem:3}
We will say $(z_1',z_2',\eta'):=(\mathcal Y^*z_1, \mathcal Y^*z_2, \mathcal Y^*\eta)=(z_1+r_1,z_2+r_2,\eta+r_\eta)$ is the "original" coordinate and $(z_1,z_2,\eta)$ is the "new" coordinate.
\end{remark}

We set the pull-back of the energy by $K$.
That is, we set
\begin{align*}
K(z_1,z_2,\eta)=E\circ \mathcal Y(z_1,z_2,\eta)=E(z_1+r_1,z_2+r_2,\eta+r_\eta).
\end{align*}
Now, we define the Hamiltonian vector field associated to $F$ with respect to the symplectic form $\Omega_0$.
We define $X_F$ by
\begin{align*}
\Omega_0(X_F,Y)=\<\nabla F, Y\>.
\end{align*}
Further, set $(X_F)_{j,A}:=dz_{j,A}X_F$ for $j=1,2$ and $A=R,I$ and $(X_F)_\eta:=d\eta(X_F)$.
Then, if $u$ is a solution of \eqref{1}, $z_{j,A}$ and $\eta$ satisfies
\begin{align*}
\dot z_{j,A} = (X_K)_{j,A},\quad j=1,2,\ A=R,I,\quad
\dot \eta = (X_K)_\eta.
\end{align*}
We now directly compute $(X_K)_\eta$.
By the definition of $X_K$, we have
\begin{align*}
\Omega_0(X_K,Y)=\sum_{j,A}\sum_{k,B}\Omega(D_{k,B}\Psi,D_{j,A}\Psi)(X_K)_{k,B}dz_{j,A}Y+\Omega((X_K)_\eta, d\eta Y),
\end{align*}
and
\begin{align*}
\<\nabla K, Y\>=\sum_{j,A}D_{j,A}F dz_{j,A}Y +\<\nabla_\eta F, d\eta Y\>.
\end{align*}
Therefore, we have
\begin{align}\label{30.51}
\dot \eta =(X_K)_\eta=-\im \nabla_\eta K.
\end{align}
We will postpone the computation of $(X_K)_{j,A}$.

Our next task is to compute the pull-back of the energy $K$.
Before computing, we make one observation.
The following lemma is corresponds to Lemma 4.11 (Cancellation Lemma) of \cite{CuMaAPDE}.
\begin{lemma}\label{lem:14}
Let $\delta>0$ sufficiently small.
Then, for any $(z_1,z_2)\in B_{\C^2}(\delta)$, we have
\begin{align*}
\nabla_\eta K(z_1,z_2,0)=0.
\end{align*}
\end{lemma}

\begin{proof}
First, notice that if $\eta=0$, from Lemma \ref{lem:13}, we have $r_1=r_2=0$ and $r_\eta=0$.
Therefore, the new and original coordinate corresponds in this case.

Next, recall that if we have the initial condition $u(z_1',z_2',0)=\Psi(z_1',z_2')$, since $\Psi$ is the quasi-periodic solution, $\eta'$ will always be $0$.
Therefore, the new and original coordinate will correspond for all time and further, we will have $\eta=0$ for all time.

Now, suppose $\nabla_\eta K(z_1,z_2,0)\neq 0 $.
Then, from \eqref{30.51}, we have
\begin{align}\label{30.6}
\im \left.\frac{d}{dt}\right|_{t=0} \eta =\nabla _\eta K(z_1,z_2,0)\neq 0.
\end{align}
However, l.h.s.\ of \eqref{30.6} is $0$ because $\eta(t)\equiv 0$.
Therefore, we have the conclusion.
\end{proof}

%\begin{proof}[Alternative proof]
%Set $p_\Theta(s)=(z_1,z_2,s\Theta)$.
%Then, we have
%\begin{align*}
%\<\nabla_\eta K(z_1,z_2,0),\Theta\>&=\left.\frac{d}{ds}\right|_{s=0}K(u(p_\Theta(s)))\\&
%=\<\nabla E(\Psi(z_1,z_2)),\left.\frac{d}{ds}\right|_{s=0}u(z_0'(p_\Theta(s)),z_1'(p_\Theta(s)),\eta'(p_\Theta(s)))\>\\&
%=\mathcal E_j\<\mathcal Y_* \frac{\partial}{\partial \theta_j},\mathcal Y_*\Theta\>=\mathcal E_j \Omega_0(\frac{\partial}{\partial \theta_j},\Theta)=0.
%\end{align*}
%\end{proof}

We prepare another lemma before computing $K$.

\begin{lemma}\label{lem:15}
Set
\begin{align*}
&\delta \Psi(z_1,z_2,\eta):=\Psi(z_1+r_1(z_1,z_2,\eta),z_2+r_2(z_1,z_2,\eta))-\Psi(z_1,z_2),\\
& \delta\eta(z_1,z_2,\eta):=R[z_1+r_1(z_1,z_2,\eta),z_2+r_2(z_1,z_2,\eta)](\eta+r_\eta(z_1,z_2,\eta))-\eta.
\end{align*}
Then, we have
\begin{align}
&\|D^{\mathbf m}\delta\Psi(z_1,z_2,\eta)\|_{l_e^{a_1}}+\|D^{\mathbf m}\delta\eta(z_1,z_2,\eta)\|_{l_e^{a_1}}\lesssim|z|^{\max(6-|\mathbf m|)}\|\eta\|_{l_e^{a_1}},\label{30.7}\\
&\|D^{\mathbf m}D_\eta\delta\Psi(z_1,z_2,\eta)\|_{\mathcal L(l_e^{-a_1}; l_e^{a_1})}+\|D^{\mathbf m}D_\eta \delta\eta(z_1,z_2,\eta)\|_{\mathcal L(l_e^{-a_1}; l_e^{a_1})}\lesssim |z|^{\max(6-|\mathbf m|)},\nonumber\\&
\|D^{\mathbf m}D_\eta^2 \delta\Psi(z_1,z_2,\eta)(\xi_1,\xi_2)\|_{l_e^{a_1}}\lesssim |z|^{\max(6-|\mathbf m|)}\|\xi_1\|_{l_e^{-a_1}}\|\xi_1\|_{l_e^{-a_2}},\nonumber
\end{align}
where $\mathbf m=(m_1,\cdots,m_{|\mathbf m|})$ with $m_j\in \{(k,B)\ |\ k=1,2,B=R,I\}$ and $D^{\mathbf m}=D_{m_1}\cdots D_{m_{|\mathbf m|}}$.

\end{lemma}

\begin{proof}
By the definition of $\delta\Psi$ and $\delta\eta$, we have
\begin{align*}
\delta \Psi(z_1,z_2,\eta) &= \sum_{j,A}\int_0^1 D_{j,A}\Psi(z_1+s r_1, z_2+s r_2)\,ds r_{j,A},\\
\delta \eta(z_1,z_2,\eta)&= r_\eta +\sum_{j,A}\<\alpha_{j,A}(z_1+r_1,z_2+r_2),\eta+r_\eta\>\phi_{j,A}.
\end{align*}
Combining the above with Lemma \ref{lem:13}, we have \eqref{30.7} with $|\mathbf m|=0$.
The estimates for the derivative respect to $D_{j,A}$ and $D_\eta$ also follows from Lemma \ref{lem:13} because of the analyticity.
\end{proof}
%%%%%%%%%%%%%            R^{a,b} とか　S^{a,b}　とか　%%%%%%%%%%%%%%%%%%%%%%

%\begin{definition}
%If $F:B_{\C^2}(0;\delta)\times B_{\mathcal H_c[0,0]}(0;\delta)\to \C$ satisfies
%\begin{align*}
%&|F(z_1,z_2,\eta)|\lesssim_\sigma (|z_0|+|z_1|+\|\eta\|_{l^{2,-\sigma}})^a\|\eta\|_{l^{2,-\sigma}}^b,\\
%&|\partial_{z_{j,A}}F(z_1,z_2,\eta)|\lesssim_\sigma (|z_0|+|z_1|+\|\eta\|_{l^{2,-\sigma}})^{\max(a-1,0)}\|\eta\|_{l^{2,-\sigma}},\\&
%\|\nabla_\eta F(z_1,z_2,\eta)\|_{l^{2,\sigma}}\lesssim_\sigma (|z_0|+|z_1|+\|\eta\|_{l^{2,-\sigma}})^a\|\eta\|_{l^{2,-\sigma}}^{\max(b-1,0)},
%\end{align*}
%we say $F(z_1,z_2,\eta)=\mathcal R^{a,b}(z_1,z_2,\eta)$.
%\end{definition}
%
%
%\begin{definition}
%If $\mathcal F:B_{\C^2}(0;\delta)\times B_{\mathcal H_c[0,0]}(0;\delta)\to l^{2,\sigma}$ satisfies
%\begin{align*}
%&\|\mathcal F(z_1,z_2,\eta)\|_{l^{2,\sigma}}\lesssim_\sigma (|z_0|+|z_1|+\|\eta\|_{l^{2,-\sigma}})^a\|\eta\|_{l^{2,-\sigma}}^b,\\
%&\|\partial_{z_{j,A}}\mathcal F(z_1,z_2,\eta)\|_{l^{2,\sigma}}\lesssim_\sigma (|z_0|+|z_1|+\|\eta\|_{l^{2,-\sigma}})^{\max(a-1,0)}\|\eta\|_{l^{2,-\sigma}},\\&
%\|d_\eta \mathcal F(z_1,z_2,\eta)\|_{l^{2,\sigma}\to l^{2,\sigma}}\lesssim_\sigma (|z_0|+|z_1|+\|\eta\|_{l^{2,-\sigma}})^a\|\eta\|_{l^{2,-\sigma}}^{\max(b-1,0)},
%\end{align*}
%we say $F(z_1,z_2,\eta)=\mathcal S^{a,b}(z_1,z_2,\eta)$.
%\end{definition}

%\begin{lemma}
%Set $\mathcal Y=\mathcal Y_1$ and $(z_1',z_2',\eta')=\mathcal Y(z_1,z_2,\eta)$.
%Then, we have
%\begin{align*}
%z_{j,A}'=z_{j,A}+\mathcal R^{6,1}(z_1,z_2,\eta),\quad \eta=\eta+\mathcal S^{6,1}(z_1,z_2,\eta).
%\end{align*}
%\end{lemma}
%
%\begin{proof}
%
%\end{proof}

We now compute the expansion of $K$.
\begin{lemma}
We have
\begin{align}\label{31}
K(z_1,z_2,\eta)=E(\Psi(z_1,z_2))+E(\eta)+ \mathcal N(z_1,z_2,\eta),
\end{align}
where $\mathcal N$ satisfies
\begin{align}
|\mathcal N(z_1,z_2,\eta)|&\lesssim (|z|^6+\|\eta\|_{l^2}^6)\|\eta\|_{l_e^{-a_2}}^2,\label{32}\\
|D_{j,A}\mathcal N(z_1,z_2,\eta)|&\lesssim (|z|^5+\|\eta\|_{l^2}^5)\|\eta\|_{l_e^{-a_2}}^2,\label{33}\\
\|\nabla_\eta \mathcal N(z_1,z_2,\eta)\|_{l_e^{a_2}}&\lesssim (|z|^6+\|\eta\|_{l^2}^6)\|\eta\|_{l_e^{-a_2}},\label{34}
\end{align}
for $a_2=a_1/3$.
\end{lemma}

\begin{proof}
By Taylor expansion, we have
\begin{align*}
E(\Psi(z_1',z_2'))&=E(\Psi(z_1,z_2))+\int_0^1\<\nabla E(\Psi(z_1,z_2)+s \delta\Psi(z_1,z_2)),\delta\Psi(z_1,z_2)\>,ds,\\
E(R[z_1',z_2']\eta')&=E(\eta)+\int_0^1\<\nabla E(\eta + s \delta\eta(z_1,z_2,\eta)),\delta\eta(z_1,z_2,\eta)\>\,ds.
\end{align*}
Therefore, by \eqref{9.02}, we have \eqref{31}, with
\begin{align}\label{35}
\mathcal N(z_1,z_2,\eta)=&\int_0^1\<\nabla E(\Psi+s \delta\Psi),\delta\Psi\>\,ds+\int_0^1\<\nabla E(\eta + s \delta\eta),\delta\eta\>\,ds\\&+\sum_{k=2}^7\sum_{\substack{i+j=k\\i\geq j}}\sum_{l+r=8-k}C_{k,i,j,l,r}\<(\Psi+\delta\Psi)^l(\overline{\Psi}+\overline{\delta\Psi})^r,(\eta+\delta\eta)^i(\overline{\eta}+\overline{\delta\eta})^j\>.\nonumber
\end{align}
It suffices to estimate each terms of r.h.s.\ of \eqref{35}.
We first estimate the second term of r.h.s.\ of \eqref{35}.
\begin{align*}
\left|\int_0^1\<\nabla E(\eta+s \delta\eta),\delta\eta\>\,ds\right|&\leq \int_0^1 \|\nabla E(\eta+ s \delta\eta)\|_{l_e^{-a_1}}\,ds\|\delta\eta\|_{l_e^{a_1}}
\lesssim \|\eta\|_{l_e^{-a_1}}\|\delta\eta\|_{l_e^{a_1}}\\&
\lesssim |z|^6\|\eta\|_{l_e^{-a_1}}^2.
\end{align*}
One can also estimate the $D_{j,A}$ derivative of this term in similar manner.
We next compute the $\nabla_\eta$ derivative.
\begin{align*}
&\<\nabla_\eta \int_0^1\<\nabla E(\eta+s \delta\eta),\delta\eta\>\,ds,\xi\>=\<\nabla_\eta\int_0^1\<H\eta + sH \delta\eta+(\eta+s \delta \eta)^4(\overline \eta+ s\overline{\delta\eta})^3,\delta\eta\>\,ds,\xi\>\\&
=\<H\xi, \delta\eta\>+\frac 1 2 \<H D_\eta \delta\eta (\xi), \delta\eta\>+4\int_0^1\<\(\xi+s(D_\eta \delta\eta(\xi))\)|\eta+s \delta\eta|^6, \delta\eta\>\,ds\\&
+3\int_0^1\<\(\overline{\xi}+s\overline{D_\eta \delta \eta (\xi)}\)|\eta+s \delta\eta|^4(\eta+ s \delta\eta)^2,\delta\eta\>\,ds+\int_0^1\<\nabla E(\eta+s \delta\eta), D_\eta \delta\eta(\xi)\>\,ds.%\\&
%H \delta \eta + \frac 1 2 (\nabla_\eta\delta \eta)^* \delta\eta+4\int_0^1(1+s(\nabla_\eta\delta\eta)^*)|\eta+s\delta\eta|^6\delta\eta\,ds\\&\quad+3\int_0^1 (1+s(\nabla_\eta \delta\eta)^*)|\eta+s \delta\eta|^4(\eta+ s \delta\eta)^2\overline{\delta\eta}\,ds+\int_0^1 (\nabla_\eta\delta\eta)^* \nabla E(\eta+s \delta\eta)\,ds.
\end{align*}
Therefore, we have
\begin{align*}
&\|\nabla_\eta \int_0^1\<\nabla E(\eta+s \delta\eta),\delta\eta\>\,ds\|_{l_e^{a_1}}\lesssim \|\delta\eta\|_{l_e^{a_1}}+\|D_\eta \delta\eta\|_{\mathcal L(l_e^{-a_1};l_e^{a_1})}\|\delta\eta\|_{l_e^{a_1}}\\&\quad+(1+\|D_\eta \delta\eta\|_{\mathcal L(l_e^{-a_1};l_e^{a_1})})(\|\eta\|_{l^\infty}^6+\|\delta\eta\|_{l^\infty}^6)\|\delta\eta\|_{l_e^{a_1}}+\|D_\eta \delta\eta\|_{\mathcal L(l_e^{-a_1};l_e^{a_1})}\int_0^1\|\nabla E(\eta+s \delta\eta )\|_{l_e^{-a_1}}\,ds\\&\lesssim
|z|^6\|\eta\|_{l_e^{-a_1}}.
\end{align*}
The third term of \eqref{35} can be bounded in similar manner.
However, for example, the estimate of $\<\Psi,|\eta|^7\eta\>$, we have
\begin{align*}
|\<\Psi,|\eta|^7\eta\>|\leq |z|\|\eta\|_{l^\infty}^5 \|\eta\|_{l_e^{a_1/3}}^2.
\end{align*}
This is why we have to make $a_1$ smaller and replace $|z|^j$ to $|z|^j+\|\eta\|_{l^2}^j$.

We finally estimate the first term of \eqref{35}.
Expanding $\nabla E(\Psi+s \delta\Psi)$, we have
\begin{align*}
\int_0^1\<\nabla E(\Psi+s \delta\Psi),\delta\Psi\>\,ds=&\quad\<\nabla E(\Psi), \delta\Psi\>+\frac 1 2 \<H \delta\Psi,\delta\Psi\>\\&+\int_0^1\<|\Psi+s \delta\Psi|^6(\Psi+\delta\Psi)-|\Psi|^6\Psi, \delta\Psi\>\,ds.
\end{align*}
The last two terms, which has at least two $\delta\Psi$ can be estimated as before.
Now, notice that the only possible source of the first order term of $\eta$ is $\<\nabla E(\Psi), \delta\Psi\>$.
However, by Lemma \ref{lem:14}, for arbitrary $\xi\in l_c^2$, we have
\begin{align*}
0=\<\nabla_\eta K(z_1,z_2,0),\xi\>=\<\nabla_\eta\<\nabla E(\Psi),\delta \Psi(z_1,z_2,0)\>,\xi\>=\<\nabla E(\Psi),D_\eta \delta\Psi(z_1,z_2,0)(\xi)\>
\end{align*}
Therefore, by Taylor expansion, we have
\begin{align*}
&\<E(\Psi(z_1,z_2)),\delta\Psi(z_1,z_2,\eta)\>=\int_0^1(1-s)\<E(\Psi),D_\eta^2\delta\Psi(z_1,z_2,s\eta)(\eta,\eta)\>\,ds,\\
&\<\nabla_\eta \<E(\Psi(z_1,z_2)),\delta\Psi(z_1,z_2,\eta)\>,\xi\>=\int_0^1\<E(\Psi),D_\eta^2\delta\Psi(z_1,z_2,s\eta)(\eta,\xi)\>\,ds.
\end{align*}
Thus, by Lemma \ref{lem:15}, we have
\begin{align*}
&|\<E(\Psi),\delta\Psi\>|\lesssim \sup_{s\in[0,1]}\|D_\eta^2\delta\Psi(z_1,z_2,s\eta)(\eta,\eta)\|_{l_e^{a_1}}\lesssim |z|^6 \|\eta\|_{l_e^{-a_1}}^2,\\&
\|\nabla_\eta\<E(\Psi),\delta\Psi\>\|_{l_e^{a_1}}\lesssim \sup_{s\in[0,1]}\|D_\eta \delta\Psi(z_1,z_2,s\eta)(\eta,\cdot)\|_{\mathcal L(l_e^{-a_1}; l_e^{a_1})}\lesssim |z|^6 \|\eta\|_{l_e^{-a_1}}.
\end{align*}
The estimate for $D_{j,A}\<E(\Psi),\delta\Psi\>$ can be obtained by similar manner.
Therefore, we have the conclusion.
\end{proof}

%%%%%%%%%%%%%%%  z の等式をポアソンブラケットなしでだす手順　%%%%%%%%%%%%%%%%%%%%%
%First, we compute the Hamiltonian vector field $X_F$.
%Set
%\begin{align*}
%\mathcal A(z_1,z_2)=\begin{pmatrix}
%a_{0,R,0,R}(z_1,z_2) & a_{0,I,0,R}(z_1,z_2) & a_{1,R,0,R}(z_1,z_2) & a_{1,I,0,R}(z_1,z_2) \\
%a_{0,R,0,I}(z_1,z_2) & a_{0,I,0,I}(z_1,z_2) & a_{1,R,0,I}(z_1,z_2) & a_{1,I,0,I}(z_1,z_2) \\
%a_{0,R,1,R}(z_1,z_2) & a_{0,I,1,R}(z_1,z_2) & a_{1,R,1,R}(z_1,z_2) & a_{1,I,1,R}(z_1,z_2) \\
%a_{0,R,1,I}(z_1,z_2) & a_{0,I,1,I}(z_1,z_2) & a_{1,R,1,I}(z_1,z_2) & a_{1,I,1,I}(z_1,z_2) 
%\end{pmatrix},
%\end{align*}
%where
%$a_{j,A,k,B}(z_1,z_2)=\Omega(D_{j,A}\Psi(z_1,z_2),D_{k,B}\Psi(z_1,z_2))$.
%Then, we have
%\begin{align*}
%\mathcal A(z_1,z_2)=
%\begin{pmatrix}
%0 & -1 & 0 & 0 \\
%1 & 0 & 0 & 0 \\
%0 & 0 & 0 & -1 \\
%0 & 0 & 1 & 0 
%\end{pmatrix}+o(1),
%\end{align*}
%Therefore, $\mathcal A(z_1,z_2)$ is invertible for $|z_0|+|z_1|$ sufficiently small.

We now try to obtain the equations which $z_{j}$ satisfies.
Set
\begin{align*}
\{F,G\}:=dF(u)X_G(u)=\<\nabla F(u), X_G(u)\>=\Omega(X_F(u),X_G(u))
\end{align*}
for $F:l^2\to \C$.
Then, if $u$ is a solution of \eqref{1}, we have
\begin{align*}
\frac{d}{dt}F(u)=\{F,K\},%\quad \frac{d}{dt}\mathcal F(u)=\{\mathcal F, E\}.
\end{align*}
Therefore, setting
\begin{align*}
K(z_1,z_2,\eta)=K_0(z_1,z_2)+K_1(z_1,z_2,\eta),
\end{align*}
where $K_0(z_1,z_2)=E(\Psi(z_1,z_2))$.
we have
\begin{align*}
\dot z_j =\{z_j,K_0\}+\{z_j, K_1\}.
\end{align*}
%We compute $\{z_j,K_n\}$ for $n=0,1$.
Now, since
\begin{align*}
\Omega_0(X_{K_n},Y)=dK_n(u)Y=\sum_{j=1,2,A=R,I}\partial_{z_{j,A}}K_n Y_{j,A}+\<\nabla _\eta K_n, Y_\eta\>,,
\end{align*}
and
\begin{align*}
\Omega_0(X_{K_n},Y)=\sum_{j,k=1,2,A,B=R,I}\Omega(D_{k,B}\Psi(z_1,z_2),D_{j,A}\Psi(z_1,z_2))(X_{K_n})_{k,B}(Y)_{j,A}+\Omega((X_{K_n})_\eta,Y_\eta),
\end{align*}
we have
\begin{align}\label{36}
(X_F)_z=\mathcal A(z_1,z_2)^{-1}\partial_z F,
\end{align}
where, 
\begin{align*}
(X_F)_z=
\begin{pmatrix}
(X_F)_{1,R}\\
(X_F)_{1,I}\\
(X_F)_{2,R}\\
(X_F)_{2,I}
\end{pmatrix},\quad
\partial_z F=
\begin{pmatrix}
D_{1,R}F\\
D_{1,I}F\\
D_{2,R}F\\
D_{2,I}F
\end{pmatrix},
\end{align*}
and
\begin{align*}
\mathcal A(z_1,z_2)=\begin{pmatrix}
a_{1,R,1,R}(z_1,z_2) & a_{1,I,1,R}(z_1,z_2) & a_{2,R,1,R}(z_1,z_2) & a_{2,I,1,R}(z_1,z_2) \\
a_{1,R,1,I}(z_1,z_2) & a_{1,I,1,I}(z_1,z_2) & a_{2,R,1,I}(z_1,z_2) & a_{2,I,1,I}(z_1,z_2) \\
a_{1,R,2,R}(z_1,z_2) & a_{1,I,2,R}(z_1,z_2) & a_{2,R,2,R}(z_1,z_2) & a_{2,I,2,R}(z_1,z_2) \\
a_{1,R,2,I}(z_1,z_2) & a_{1,I,2,I}(z_1,z_2) & a_{2,R,2,I}(z_1,z_2) & a_{2,I,2,I}(z_1,z_2) 
\end{pmatrix},
\end{align*}
where
$a_{j,A,k,B}(z_1,z_2)=\Omega(D_{j,A}\Psi(z_1,z_2),D_{k,B}\Psi(z_1,z_2))$.
Notice that since
\begin{align*}
\mathcal A(z_1,z_2)=
\begin{pmatrix}
0 & -1 & 0 & 0 \\
1 & 0 & 0 & 0 \\
0 & 0 & 0 & -1 \\
0 & 0 & 1 & 0 
\end{pmatrix}+o(1),
\end{align*}
$\mathcal A(z_1,z_2)$ is invertible

We will not compute $\{z_j,K_0\}$ directly but use the fact that $\Psi(z_1,z_2)$ is the solution of \eqref{1}.

\begin{lemma}
We have
\begin{align*}
\{z_j, K_0(z_1,z_2)\}=-\im \mathcal E_j(|z_1|^2,|z_2|^2)z_j
\end{align*}
\end{lemma}

\begin{proof}
First, if $\eta=0$, we have $z_0'=z_0$ and $z_1'=z_1$.
Therefore, since $\Psi(z_1,z_2)$ is the solution \eqref{1} if $\im \dot z_j = \mathcal E(|z_1|^2,|z_2|^2) z_j$, we have
\begin{align*}
-\im \E(|z_1|^2,|z_2|^2)z_j=\dot z_j =\{ z_j, K (z_1,z_2,\eta)\}_{\eta=0}=\{ z_j, K_0 (z_1,z_2)\}_{\eta=0}+\{ z_j, K_1 (z_1,z_2,\eta)\}_{\eta=0}.
\end{align*}
On the other hand, from \eqref{36}, we see that $\{ z_j, K_1 (z_1,z_2,\eta)\}_{\eta=0}=0$ because it consists from the $D_{k,B}$ derivative or $K_1$ which is $0$ if $\eta=0$.
Therefore, we have
\begin{align*}
-\im \E(|z_1|^2,|z_2|^2)z_j=\{ z_j, K_0 (z_1,z_2)\}_{\eta=0}.
\end{align*}
Finally, since the symplectic form $\Omega_0$ do not depend on $\eta$ (although it depends on $z_j$), we have the conclusion.
\end{proof}

We set $R_j=\{z_j,K_1(z_1,z_2,\eta)\}$.
Then, by \eqref{36}, we have $R_j=\{z_j, \mathcal N(z_1,z_2,\eta)\}$.
Futher, combining \eqref{36} with \eqref{33}, we have
\begin{align}\label{37}
|R_j|\lesssim (|z|^5+\|\eta\|_{l^2})\|\eta\|_{l_e^{-a_1}}.
\end{align}

As a conclusion of this section, we have the equations of $z_j$ and $\eta$.
\begin{align}
&\im \eta_t = H\eta + P_c\(|\eta|^6\eta + \nabla_\eta\mathcal N\),\label{40}\\&
\dot z_j=-\im \E(|z_1|^2,|z_2|^2)z_j+R_j,\quad j=1,2\label{41}.
\end{align}

\section{Linear estimates}\label{sec:linest}

In this section, we introduce the linear estimates for the proof of Theorem \ref{thm:2}.
Lemmas \ref{lem:l0}--\ref{lem:l3} can be found in \cite{CT09SIAM}.
See also \cite{PS08JMP} and \cite{KPS09SIAM}.
In the following we always assume $H$ is generic in the sense of Lemma 5.3 of \cite{CT09SIAM}.

\begin{definition}
We say the pair of numbers $(r,p)$ is admissible if
\begin{align*}
\frac 2 r + \frac 1 p = \frac 1 2 ,\quad (r,p)\in [4,\infty]\times [2,\infty].
\end{align*}
We set
\begin{align*}
X_{r,p}:=l^{\frac 3 2 r}(\Z, L_t^\infty([n,n+1], l^p)),\quad X_{r,p}'=l^{\(\frac 3 2 r\)'}(\Z, L_t^1([n,n+1], l^{p'})),
\end{align*}
where $p'$ is the H\"older conjugate of $p$ (i.e. $\frac 1 p + \frac{1}{p'}=1$).
\end{definition}

%In the following, we always assume $H$ is generic and $V\in l^{1,1}$.

\begin{lemma}[Dispersive estimate]\label{lem:l0}
We have
\begin{align*}
\|e^{-\im t H}P_c\|_{\mathcal L(l^1;l^\infty)}\lesssim \<t\>^{-1/3}.
\end{align*}
\end{lemma}

\begin{lemma}[Strichartz estimate]\label{lem:l1}
Let $(p,r)$, $(p_1,r_1)$ and $(p_2,r_2)$ admissible.
Then, we have
\begin{align*}
\|e^{-\im t H}P_c f\|_{X_{r,p}}\lesssim \|f\|_{l^2},
\end{align*}
and
\begin{align*}
\left\|\int_0^t e^{-\im (t-s)H}P_cg(s)\,ds\right\|_{X_{r_1,p_1}}\lesssim \|g\|_{X_{r_2,p_2}'}.
\end{align*}
\end{lemma}

\begin{lemma}[Kato Smoothing]\label{lem:l2}
Let $\sigma>1$.
Then, we have
\begin{align*}
\|e^{-\im t H}P_c f\|_{L^tl^{2,-\sigma}}\lesssim \|f\|_{l^2},
\end{align*}
and
\begin{align*}
\left\| \int_0^te^{-\im(t-s)H}P_c g(s)\,ds\right\|_{L^2_tl^{2,-\sigma}}\lesssim\|g\|_{L^2l^{2,\sigma}}.
\end{align*}
\end{lemma}

\begin{lemma}\label{lem:l3}
Let $\sigma>1$ and $(r,p)$ admissible.
Then, we have
\begin{align*}
\left\|\int_0^t e^{-\im(t-s)H}P_cg(s,\cdot)\,ds\right\|_{X_{r,p}}\lesssim \|g\|_{L^2_t l^{2,\sigma}}.
\end{align*}
\end{lemma}

\section{Proof of Theorem \ref{thm:2}}\label{sec:proofthm2}

We are now in the position to prove Theorem \ref{thm:2} and Corollary \ref{cor:1}.

Fix $\sigma>1$ and set $X=X_{4,\infty}\cap X_{\infty,2}\cap L^2l^{2,-\sigma}$.
\begin{proposition}\label{prop:4}
Under the hypothesis of Theorem \ref{thm:2}, there exists $\epsilon_0>0$ s.t. if 
$\|u_0\|_{l^2}=\epsilon<\epsilon_0$, we have
\begin{align}
\|\eta\|_{X}&\lesssim \|\eta(0)\|_{l^2},\label{50}\\
\left\|\frac{d}{dt}|z_j|^2\right\|_{L^1}&\lesssim \epsilon^6\|\eta(0)\|_{l^2}^2,\quad j=1,2.\label{51}
\end{align}
\end{proposition}

\begin{proof}
First, by the $l^2$ conservation of \eqref{1} and \eqref{30.1} of Lemma \ref{lem:13}, we have 
\begin{align*}
|z_1|+|z_2|+\|\eta\|_{l^2}&\lesssim \sum_{j=1,2}|z_j'|+|r_j(z_1,z_2,\eta)|+\|\eta\|_{l^2}+\|r_\eta(z_1,z_2,\eta)\|_{l^2}\\&
\lesssim \epsilon+|z|^6\|\eta\|_{l^2}.
\end{align*}
Therefore, we have
\begin{align}\label{52}
|z_1|+|z_2|+\|\eta\|_{l^2}\lesssim \epsilon,
\end{align}
for all time $t$.
By \eqref{40},  
for any admissible pair $(r,p)$, we have
\begin{align*}
\|\eta\|_{X_{r,p}}&\leq \|e^{-\im t H}\eta(0)\|_{X_{r,p}}+\|\int_0^t e^{-\im(t-s)H}P_c \nabla_\eta\mathcal N\,ds\|_{X_{r,p}}+\|\int_0^t e^{-\im(t-s)H}P_c \mathcal |\eta|^6\eta\,ds\|_{X_{r,p}}\\&
\lesssim \|\eta(0)\|_{l^2}+\|\nabla_\eta\mathcal N\|_{L^2l^{2,\sigma}}+\| |\eta|^6\eta \|_{L^1l^2}\\&
\lesssim \|\eta(0)\|_{l^2}+ \epsilon^6 \|\eta\|_{L^2l^{2,-\sigma}}+ \|\eta\|_{L^\infty l^2}\|\eta\|_{X_{4,\infty}}^6,
\end{align*}
where we have used Lemma \ref{lem:l1} and \ref{lem:l3} in the first inequality and \eqref{34} in the second inequality.
Again by \eqref{40} and Lemma \ref{lem:l1}, \ref{lem:l2}, we have
\begin{align*}
\|\eta\|_{L^2l^{2,-\sigma}}&\lesssim \|\eta(0)\|_{l^2} + \|\nabla_\eta\mathcal N\|_{L^2l^{2,\sigma}}+\int_0^\infty \||\eta|^7\|_{l^2}\,ds\\&
\lesssim \|\eta(0)\|_{l^2} + \epsilon^6\|\eta\|_{L^2l^{2,-\sigma}}+\|\eta\|_{X_{14/3,14}}^7,
\end{align*}
where we have use $\|\eta\|_{L^7l^{14}}\leq \|\eta\|_{X_{14/3,14}}$ in the second inequality.
Therefore, we have
\begin{align*}
\|\eta\|_X\lesssim \|\eta(0)\|_{l^2}+\epsilon^6\|\eta\|_X+\|\eta\|_X^7.
\end{align*}
By continuity argument, we have \eqref{50}.

Next, multiplying $\overline{z_j}$ to \eqref{41} and taking the real part, we have
\begin{align*}
\frac{d}{dt}|z_j|^2=R_j\overline{z_j}+\overline{R_j}z_j.
\end{align*}
Therefore, by \eqref{37}, we have
\begin{align*}
\|\frac{d}{dt}|z_j|^2\|_{L^1}\leq (\|z_1\|_{L^\infty}^6+\|z_2\|_{L^\infty}^6+\|\eta\|_{L^\infty l^2}^6)\|\eta\|_{L^2 l^{2,-\sigma}}^2\lesssim \epsilon^6\|\eta\|_{L^2 l^{2,-\sigma}}^2.
\end{align*}
Combining the above with \eqref{50}, we obtain \eqref{51}.
\end{proof}

We now prove Theorem \ref{thm:2}.
\begin{proof}[Proof of Theorem \ref{thm:2}]
By Proposition \ref{prop:4}, we see that there exists $\rho_{j,+}$ and $v_+$ s.t.
\begin{align*}
|z_j(t)|\to \rho_{j,+},\quad \mathrm{and}\quad \|\eta(t)-e^{\im t \Delta}\eta_+\|_{l^2}\to 0.
\end{align*}
with $\rho_{0,+}+\rho_{1,+}+\|v_+\|_{l^2}\lesssim \epsilon$.

Now, by Lemma \ref{lem:l0} we have $\|\eta(t)\|_{l_e^{-a}}\to 0$ for any $a>0$.
Therefore, by Lemma \ref{lem:13}, we see $\|\eta'(t)-\eta(t)\|_{l^2}\to 0$ and $\left||z_j(t)|-|z_j'(t)|\right|\to 0$ as $t\to \infty$.
Here, $(z'_1,z_2',\eta')$ are the original coordinates (see remark \ref{rem:3}).
Therefore, we have the conclusion.
\end{proof}

We next prove Corollary \ref{cor:1}.

\begin{proof}[Proof of Corollary \ref{cor:1}]
Fix $j\in \{1,2\}$ and fix $z_j\in \C$ with $|z_j|\ll1$.
Now, let $0\ll \epsilon\ll|z_j|$ and assume $\|u(0)-\phi_2(z_j)\|_{l^2}\lesssim \epsilon$, then we have $|z_{3-j}(0)|+|z_j-z_j(0)|+\|\eta(0)\|_{l^2}\lesssim \epsilon$.
Further, by Proposition \ref{prop:4}, we have
\begin{align*}
\sup_{t\geq 0}\(|z_{3-j}(t)|^2+\|\eta(t)\|_{l^2}\)\lesssim |z_{3-j}(0)|^2+\|\eta(0)\|_{l^2}^2,
\end{align*}
and
\begin{align*}
\sup_{t\geq 0}\(|z_j(t)|^2-|z_j|^2\)\lesssim |z_j(0)|^2-|z_j|^2+ |z_j|^6\|\eta(0)\|_{l^2}.
\end{align*}
Therefore, going back to the original coordinate, we have the conclusion.
\end{proof}

\appendix

\section{Proof of Proposition \ref{prop:1}}

In this section, we prove Proposition \ref{prop:1}.
Before proving Proposition \ref{prop:1}, we prepare an elementary estimate.
\begin{lemma}\label{lem:a1}
Let $\delta>0$.
Then there exists $a(\delta)>0$ s.t.\ for $a\in (0,a(\delta))$ and for
$\lambda\notin (-\delta,4+\delta)\cup (e_1-\delta, e_1+\delta)\cup (e_2-\delta,e_2+\delta)$,
we have
\begin{align}\label{a0}
\|(H-\lambda)^{-1}\|_{\mathcal L(l_e^a)}\lesssim_\delta \<\lambda\>^{-1}.
\end{align}
Further, let $j=1,2$.
Then, for sufficiently small $a>0$, we have
\begin{align}\label{a0.1}
\left\|\(\left.(H-e_j)\right|_{\phi_j^\perp}\)^{-1}\right\|_{\mathcal L(l_e^{a})}\lesssim 1.
\end{align}
\end{lemma}

\begin{proof}
Set $T_{a,N}$ by
\begin{align*}
(T_{a,N}v)(n)=e^{a \min(|n|,N)}v(n).
\end{align*}
We first claim there exists $B_{a,N}:l^2\to l^2$ s.t.\ $\|B_{a,N}\|_{l^2\to l^2}\lesssim a$ (the implicit constant do not depend on $N$) and
\begin{align*}
T_{a,N}(H-\lambda)T_{a,N}^{-1}=H-\lambda + B_{a,N}.
\end{align*}
Indeed, setting $B_{a,N}=T_{a,N}(-\Delta)T_{a,N}^{-1}+\Delta$, we have
\begin{align*}
(B_au)(n)=\(1-e^{a(\min(|n|,N)-\min(|n+1|,N))}\)u(n+1)+\(1-e^{a(\min(|n|,N)-\min(|n-1|,N))}\)u(n-1).
\end{align*}
Since $|1-e^{a(\min(|n|,|N|)-\min(|n+1|,N))}|\lesssim a$ and $|1-e^{a(\min(|n|,N)-\min(|n-1|,N))}|\lesssim a$, we have the desired bound for $B_{a,N}$.
Now, since
\begin{align*}
T_{a,N}(H-\lambda)^{-1}T_{a,N}^{-1}&=(T_{a,N}(H-\lambda)T_{a,N})^{-1}=(H-\lambda+B_{a,N})^{-1}\\&=(H-\lambda)^{-1}(1+(H-\lambda)^{-1}B_{a,N})^{-1}.
\end{align*}
Therefore, by Neumann expansion and since $\|(H-\lambda)^{-1}\|_{l^2\to l^2}\lesssim \delta^{-1}$, if we take $a>0$ sufficiently small s.t.\ $a \delta^{-1}\ll 1$, we have
\begin{align*}
\|T_{a,N}(H-\lambda)^{-1}T_{a,N}^{-1}\|_{\mathcal L(l^2)}\lesssim \|(H-\lambda)^{-1}\|\lesssim_\delta \lambda^{-1}
\end{align*}
This implies that for $u\in l_e^a$,
\begin{align*}
\|T_{a,N}(H-\lambda)^{-1}u\|_{l^2}\lesssim_\delta \lambda^{-1}\| T_{a,N}u\|_{l^2}\leq \lambda^{-1}\|u\|_{l_e^a}.
\end{align*}
Taking $N\to \infty$, we obtain \eqref{a0}.

Next, we prove \eqref{a0.1}.
Suppose $u,f\perp \phi_j$ and $(H-e_j)u=f$, $u\in l^2$, $f\in l_e^{a}$.
Set $P:=\<\cdot,\phi_j\>\phi_j$ and $Q=1-P$.
Now, we have
\begin{align*}
T_{a,N}f=(H-e_j+B_{a,N})T_{a,N}u=(H-e_j+B_{a,N})(QT_{a,N}u+\<u,T_{a,N}\phi\>\phi).
\end{align*}
Therefore, we have
\begin{align*}
(H-e_j)QT_{a,N}u=T_{a,N}f-B_{a,N}QT_{a,N}u-\<u,T_{a,N}\phi\>B_{a,N}\phi.
\end{align*}
Now, by $f,\phi\in l_e^a$, where $a>0$ is sufficiently small so that $\phi_j\in l_e^a$, we have
\begin{align*}
\|QT_{a,N}u\|_{l^2}\lesssim \|f\|_{l_e^a}+a\|QT_{a,N}u\|_{l^2}+a\|u\|_{l^2}.
\end{align*}
Thus, for $a$ sufficiently small,
\begin{align*}
\|QT_{a,N}u\|_{l^2}\lesssim \|f\|_{l_e^a}+a\|u\|_{l^2},
\end{align*}
and
\begin{align*}
\|T_{a,N}u\|_{l^2}\leq \|QT_{a,N}u\|_{l^2}+\|P T_{a,N}u\|_{l^2}\lesssim \|f\|_{l_e^a}+\|u\|_{l^2}.
\end{align*}
Finally, taking $N\to \infty$, we have
\begin{align*}
\|u\|_{l_e^2}\lesssim \|f\|_{l_2^a},
\end{align*}
where we have used the fact that $\|u\|_{l^2}\lesssim \|f\|_{l^2}\leq \|f\|_{l_e^a}$.
\end{proof}

We now prove Proposition \ref{prop:1}.

\begin{proof}[Proof of Proposition \ref{prop:1}]
For simplicity, we write $\phi_j$ as $\phi$, $e_j$ as $e$ and $E_j$ as $E$.
Consider a solution in the form $z(\phi+q(|z|^2))$ with real valued $q$ with $\<\phi, q\>=0$.
Now, substitute it in the equation and we have
\begin{align*}
Hq + |z|^6 |\phi + q|^6 (\phi + q) = e q + (E-e)(\phi + q).
\end{align*}
Then, we have
\begin{align*}
&(|z|^6 |\phi + q|^6 (\phi + q) ,\phi) = E-e,\\&
Hq + Q\(|z|^6 |\phi + q|^6 (\phi + q)\) = E q.
\end{align*}
Therefore, we set
\begin{align}\label{a1}
E(|z|^2,q):=e+|z|^6\< |\phi + q|^6 (\phi + q) ,\phi\>, 
\end{align}
and we have
\begin{align*}
(H-e)q&= \(E(z,q)-e\)q-Q\(|z|^6 |\phi + q|^6 (\phi + q)\)\\&
=|z|^6\(f(q),\phi\)q-|z|^6 Q f(q),
\end{align*}
where $f(q)=|\phi + q|^6 (\phi + q)$.
We set $\mathcal F:Q l_e^a\times \R \to Q l_e^a$ by
\begin{align*}
\mathcal F(q, s):=(H-e)q -s^3\(f(q),\phi\)q+s^3 Q f(q).
\end{align*}
Then, $\mathcal F$ is real analytic with respect to $q$ and $s$.
Further, since
\begin{align*}
\left.D_{q}\mathcal F(q, s)\right|_{(q, s)=(0,0)}=H-q
\end{align*}
is invertible in $Q l_e^a$ for sufficiently small $a>0$, by implicit function theorem, for sufficiently small $s$, there exists $q(s)$ s.t. $q (s)$ is real analytic with respect to $s$ and $\mathcal F(q(s),s)=0$.
Further, comparing the Taylor series of
\begin{align*}
q(s) =s^3(H-e)^{-1}\((f(q),\phi)-Qf(q)\),
\end{align*}
we see $\|q(s)\|_{l_e^a}\lesssim s^3$.
Therefore, $q(|z|^2)$ is the desired solution.
Finally, set $E(s)=E(s,q(s))$, where the r.h.s.\ is given in \eqref{a1}.
Then, since $E(s,q)$ and $q$ is both real analytic, $E(s)$ also becomes real analytic.
The estimate $|E(|z|^2)-e|\lesssim |z|^6$ also follows from \eqref{a1}.

%We set
%\begin{align*}
%\Psi(q):=|z|^6\(\left.(H-e)\right|_{\{\phi\}^{\perp}}\)^{-1}\((f(q),\phi)q-Qf(q)\).
%\end{align*}
%It suffices to show $\Psi$ is a contraction mapping in a small ball of $l_e^a$ for $a$ and $|z|$ sufficiently small.
%First, if $q\in l_e^a$ with $\|q\|_{l_e^a}\lesssim 1$, we have
%\begin{align*}
%\|\Psi(q)\|_{l_e^a}\lesssim |z|^6\( \|f(q)\|_{l^2} \|q\|_{l_e^{a}} + \|f(q)\|_{l_e^{a}}\)
%\end{align*}
%Since $\|f(q)\|_{l_e^a}\lesssim 1$, we have
%\begin{align*}
%\|\Psi(q)\|_{l_e^a}\lesssim |z|^6.
%\end{align*}
%Therefore, if $|z|\lesssim 1$, $\Psi:B_{l^{2,1}}(0,c|z|^6)\to B_{l^{2,1}}(0,c|z|^6)$ for some constant $c>0$.
%Further, for $q_1,q_2\in l_e^a$ with $\|q_1\|_{l_e^a}+\|q_2\|_{l_e^a}\lesssim 1$, we have
%\begin{align*}
%\|\Psi(q_1)-\Psi(q_2)\|_{l_e^a}\lesssim |z|^6\(\|f(q_1)-f(q_2)\|_{l^2}\|q_1\|_{l_e^a}+\|f(q_2)\|_{l^2}\|q_1-q_2\|_{l_e^a}+\|f(q_1)-f(q_2)\|_{l_e^a}\).
%\end{align*}
%Since $\|f(q_1)-f(q_2)\|_{l_e^a}\lesssim \|q_1-q_2\|_{l_e^a}$, we have
%\begin{align*}
%\|\Psi(q_1)-\Psi(q_2)\|_{l_e^a}\lesssim |z|^6\|q_1-q_2\|_{l_e^a}.
%\end{align*}
%Therefore, $\Psi$ is a construction mapping for $|z|$ sufficiently small and
%the fixed point of $\Psi$ satisfies
%\begin{align*}
%\|q\|_{l_e^a}\lesssim |z|^6,
%\end{align*}
%and by \eqref{a1}, we have $|E(z,q)-e|\lesssim |z|^6$.

\end{proof}

\section{Proof of Lemma \ref{lem:2.0}}
\begin{proof}[Proof of Lemma \ref{lem:2.0}]
%We define $\mathbf m^{l_1l_2}(\mathbf v^1,\mathbf v^2,\mathbf v^3)=\{m^{l_1l_2}_{jm}(\mathbf v^1,\mathbf v^2,\mathbf v^3)\}_{j=1,2,m\geq 0}$ by
%\begin{align*}
%\mathcal M(|z_1|^2,|z_2|^2\mathbf v^1,\mathbf v^2,\mathbf v^3)=\sum_{l_1,l_2\geq 0}|z_1|^{2l_1}|z_2|^{2l_2}\mathbf m^{l_1l_2}(\mathbf v^1,\mathbf v^2,\mathbf v^3).
%\end{align*}
Set
$\mathbf v^k=\{v_{jm}^k\}_{j=1,2,m\geq 0}$, ($k=1,2,3$).
Then, using the relation \eqref{6.03}, we have
\begin{align*}
M_{1m}=
&\sum_{\substack{l\geq 0\\ l+m\geq m_1\geq 0}}|z_1|^{2l+2}|z_2|^{2l}v_{1m_1}^1v_{1l}^2v_{1(l+m-m_1)}^3
+\sum_{l-1\geq m_2\geq 0}|z_1|^{2l}|z_2|^{2l}v^1_{1(m+l)}v^2_{1m_2}v^3_{2(l-m_2-1)}\\&
+\sum_{\substack{m_2,m_3\geq 0\\ m_2+m_3\leq m-1}}v^1_{1(m-m_2-m_3-1)}v^2_{2m_2}v^3_{1m_3}
+\sum_{\substack{l\geq 0\\l+m\geq m_1\geq 0}}|z_1|^{2l}|z_2|^{2l+2}v^1_{1m_1}v^2_{2(l+m-m_1)}v^3_{2l}\\&
+\sum_{l-1\geq m_1\geq 0}|z_1|^{2l}|z_2|^{2l}v^1_{2m_1}v^2_{1(l-m_1-1)}v^3_{1(m+l)}
+\sum_{\substack{l\geq 0\\l+m\geq m_2\geq 0}}|z_1|^{2l}|z_2|^{2l+2}v^1_{2l}v^2_{2m_2}v^3_{1(l+m-m_2)}\\&
+\sum_{l\geq m_1\geq 0}|z_1|^{2l}|z_2|^{2l+4}v^1_{2m_1}v^2_{2(m+l+1)}v^3_{2(l-m_1)},
\end{align*}
and
\begin{align*}
M_{2m}=&
\sum_{l\geq m_1\geq 0}|z_1|^{2l+4}|z_2|^{2l}v^1_{1m_1}v^2_{1(m+l+1)}v^3_{1(l-m_1)}
+\sum_{\substack{l\geq0\\m+l\geq m_2\geq 0}}|z_1|^{2l+2}|z_2|^{2l}v^1_{1l}v^2_{1m_2}v^3_{2(m+l-m_2)}\\&
+\sum_{l-1\geq m_1\geq 0}|z_1|^{2l}|z_2|^{2l}v^1_{1m_1}v^2_{2(l-m_1-1)}v^3_{2(m+l)}
+\sum_{\substack{l\geq 0\\l+m\geq m_1\geq 0}}|z_1|^{2l+2}|z_2|^{2l}v^1_{2m_1}v^2_{1(l+m-m_1)}v^3_{1l}\\&
+\sum_{\substack{m_2,m_3\geq 0\\m_2+m_3\leq m+1}}v^1_{2(m-m_2-m_3-1)}v^2_{1m_2}v^3_{2m_3}
+\sum_{l-1\geq m_2\geq 0}|z_1|^{2l}|z_2|^{2l}v^1_{2(m+l)}v^2_{2m_2}v^3_{1(l-m_2-1)}\\&
+\sum_{\substack{l\geq 0\\l+m\geq m_1\geq 0}}|z_1|^{2l}|z_2|^{2l+2}v^1_{2m_1}v^2_{2l}v^3_{2(l+m-m_1)}.
\end{align*}
Therefore, we can express $\mathcal M$ such as
\begin{align*}
&\mathcal M(|z_1|^2,|z_2|^2,\mathbf v^1,\mathbf v^2,\mathbf v^3)=\tilde {\mathbf m}^{00}(\mathbf v^1,\mathbf v^2,\mathbf v^3)+\sum_{l\geq 0}\(|z_1|^{2(l+2)}|z_2|^{2l}\mathbf m^{(l+2)l}(\mathbf v^1,\mathbf v^2,\mathbf v^3)\right.\\&\quad\left.+|z_1|^{2(l+1)}|z_2|^{2l}\mathbf m^{(l+1)l}(\mathbf v^1,\mathbf v^2,\mathbf v^3)+|z_1|^{2l}|z_2|^{2l}\mathbf m^{ll}(\mathbf v^1,\mathbf v^2,\mathbf v^3)\right.\\&\left.\quad +|z_1|^{2l}|z_2|^{2(l+1)}\mathbf m^{l(l+1)}(\mathbf v^1,\mathbf v^2,\mathbf v^3)+|z_1|^{2l}|z_2|^{2(l+2)}\mathbf m^{l(l+2)}(\mathbf v^1,\mathbf v^2,\mathbf v^3)\),
\end{align*}
where $\tilde{\mathbf m}^{00}=\{\tilde m^{00}_{jm}\}_{j=1,2,m\geq 0}$ and $\mathbf m^{l_1l_2}=\{m^{l_1l_2}_{jm}\}_{j=1,2,m\geq 0}$ are given by
\begin{align*}
&\tilde m^{00}_{1m}=\sum_{\substack{m_2,m_3\geq 0\\ m_2+m_3\leq m-1}}v^1_{1(m-m_2-m_3-1)}v^2_{2m_2}v^3_{1m_3},\quad \tilde m^{00}_{2m}=\sum_{\substack{m_2,m_3\geq 0\\m_2+m_3\leq m+1}}v^1_{2(m-m_2-m_3-1)}v^2_{1m_2}v^3_{2m_3},
\end{align*}
and
\begin{align*}
&m^{(l+2)l}_{1m}=0,\quad
m^{(l+2)l}_{2m}=\sum_{l\geq m_1\geq 0}v^1_{1m_1}v^2_{1(m+l+1)}v^3_{1(l-m_1)},\quad
m^{(l+1)l}_{1m}=\sum_{\substack{l\geq 0\\ l+m\geq m_1\geq 0}}v_{1m_1}^1v_{1l}^2v_{1(l+m-m_1)}^3,\\&
m^{(l+1)l}_{2m}=
\sum_{\substack{l\geq0\\m+l\geq m_2\geq 0}}v^1_{1l}v^2_{1m_2}v^3_{2(m+l-m_2)}
+\sum_{\substack{l\geq 0\\l+m\geq m_1\geq 0}}v^1_{2m_1}v^2_{1(l+m-m_1)}v^3_{1l},\\
&m^{ll}_{1m}=\sum_{l-1\geq m_2\geq 0}v^1_{1(m+l)}v^2_{1m_2}v^3_{2(l-m_2-1)}+
\sum_{l-1\geq m_1\geq 0}v^1_{2m_1}v^2_{1(l-m_1-1)}v^3_{1(m+l)},\\&
m^{ll}_{2m}=
\sum_{l-1\geq m_1\geq 0}v^1_{1m_1}v^2_{2(l-m_1-1)}v^3_{2(m+l)}
+\sum_{l-1\geq m_2\geq 0}v^1_{2(m+l)}v^2_{2m_2}v^3_{1(l-m_2-1)},\\&
m^{l(l+1)}_{1m}=\sum_{\substack{l\geq 0\\l+m\geq m_1\geq 0}}v^1_{1m_1}v^2_{2(l+m-m_1)}v^3_{2l}
+\sum_{\substack{l\geq 0\\l+m\geq m_2\geq 0}}v^1_{2l}v^2_{2m_2}v^3_{1(l+m-m_2)},\\&
m^{l(l+1)}_{2m}=\sum_{\substack{l\geq 0\\l+m\geq m_1\geq 0}}v^1_{2m_1}v^2_{2l}v^3_{2(l+m-m_1)},\quad
m^{l(l+2)}_{1m}=\sum_{l\geq m_1\geq 0}v^1_{2m_1}v^2_{2(m+l+1)}v^3_{2(l-m_1)},\quad
m^{l(l+2)}_{2m}=0.
\end{align*}
Now, we can estimate $\|\tilde {\mathbf m}^{00}\|_{ar}$ as
\begin{align*}
&\|\tilde{\mathbf m}^{00}(\mathbf v_1,\mathbf v_2,\mathbf v_3)\|_{ar}=\sum_{j=1,2,m\geq 0}r^{2m+1}\|\tilde m^{00}_{jm}(\mathbf v_1,\mathbf v_2,\mathbf v_3)\|_{l_e^a}\leq 
\sum_{j=1,2,m\geq 0}r^{2m+1}\\&\times\(\sum_{\substack{m_2,m_3\geq 0\\ m_2+m_3\leq m-1}}\|v^1_{1(m-m_2-m_3-1)}v^2_{2m_2}v^3_{1m_3}\|_{l_e^a}+\sum_{\substack{m_2,m_3\geq 0\\m_2+m_3\leq m+1}}\|v^1_{2(m-m_2-m_3-1)}v^2_{1m_2}v^3_{2m_3}\|_{l_e^a}\)\\&\leq \|\mathbf v^1\|_{ar}\|\mathbf v^2\|_{ar}\|\mathbf v^3\|_{ar}.
\end{align*}
Similarly, we have
\begin{align*}
\|\mathbf m^{(l+2)l}(\mathbf v_1,\mathbf v_2,\mathbf v_3)\|_{ar}&\leq r^{-(4l+4)} \|\mathbf v^1\|_{ar}\|\mathbf v^2\|_{ar}\|\mathbf v^3\|_{ar},\\
\|\mathbf m^{(l+1)l}(\mathbf v_1,\mathbf v_2,\mathbf v_3)\|_{ar}&\leq 3 r^{-(4l+2)} \|\mathbf v^1\|_{ar}\|\mathbf v^2\|_{ar}\|\mathbf v^3\|_{ar},\\
\|\mathbf m^{ll}(\mathbf v_1,\mathbf v_2,\mathbf v_3)\|_{ar}&\leq 4 r^{-4l} \|\mathbf v^1\|_{ar}\|\mathbf v^2\|_{ar}\|\mathbf v^3\|_{ar},\\
\|\mathbf m^{l(l+1)}(\mathbf v_1,\mathbf v_2,\mathbf v_3)\|_{ar}&\leq 3 r^{-(4l+2)} \|\mathbf v^1\|_{ar}\|\mathbf v^2\|_{ar}\|\mathbf v^3\|_{ar},\\
\|\mathbf m^{l(l+2)}(\mathbf v_1,\mathbf v_2,\mathbf v_3)\|_{ar}&\leq  r^{-(4l+4)} \|\mathbf v^1\|_{ar}\|\mathbf v^2\|_{ar}\|\mathbf v^3\|_{ar}.
\end{align*}
Therefore, we see that $\mathcal M$ uniformly converges in $\mathcal L^3(X_{ar};X_{ar})$ if $|z|\leq \delta<r$.
Thus, we have the conclusion of Lemma \ref{lem:2.0} for $k=1$.
The cases $k\geq 2$ follow from the inductive definition of $\mathcal M_{2k+1}$.
\end{proof}

\subsection*{Acknowledgments}   
The author appreciate S.\ Cuccagna for helpful comments.
The author also would like to thank D.\ Pelinovsky and D.\ Bambusi for pointing out related works.
The author was supported by the Japan Society for the Promotion of Science (JSPS) with the Grant-in-Aid for Young Scientists (B) 15K17568.

%\bibliographystyle{amsplain}
%\bibliography{Bib}
%

Department of Mathematics and Informatics,
Faculty of Science,
Chiba University,
Chiba 263-8522, Japan

{\it E-mail Address}: {\tt maeda@math.s.chiba-u.ac.jp}

\end{document}